\documentclass[review,hidelinks,onefignum,onetabnum]{siamart220329}

\usepackage{graphics} 
  \usepackage{epsfig}
\usepackage{lipsum}
\usepackage{amsfonts}
\usepackage{amssymb}
\usepackage{graphicx}
\usepackage{epstopdf}
\usepackage{algorithmic}
\usepackage{subfigure}
\ifpdf
  \DeclareGraphicsExtensions{.eps,.pdf,.png,.jpg}
\else
  \DeclareGraphicsExtensions{.eps}
\fi

\usepackage{enumitem}
\setlist[enumerate]{leftmargin=.5in}
\setlist[itemize]{leftmargin=.5in}

\newsiamremark{remark}{Remark}
\newsiamremark{example}{Example}
\newsiamremark{hypothesis}{Hypothesis}
\newsiamremark{assumption}{Assumption}
\crefname{hypothesis}{Hypothesis}{Hypotheses}
\newsiamthm{claim}{Claim}

\ifpdf
\hypersetup{
  pdftitle={An Optimal Control Problem with uncertain terminal complementarity constraints},
  pdfauthor={J. Luo and X. Chen}
}
\fi

\headers{ An Optimal Control Problem with Complementarity Constraints }{J. LUO and X. CHEN}

\title{An Optimal Control Problem with  Terminal Stochastic Linear Complementarity Constraints\thanks{Submitted to the editors DATE.
\funding{This work is supported by the Hong Kong Research Grants Council grant PolyU15300021 and the CAS AMSS-PolyU Joint Laboratory in Applied Mathematics.}}}

\author{Jianfeng Luo\thanks{CAS AMSS-PolyU Joint Laboratory of Applied Mathematics (Shenzhen), The Hong Kong Polytechnic University Shenzhen Research Institute, Shenzhen, China
  (\email{gstljf@polyu.edu.hk}).}
\and Xiaojun Chen\thanks{Department of Applied Mathematics, The Hong Kong Polytechnic University, Hong Kong, China (\email{maxjchen@polyu.edu.hk}).}}

\usepackage{amsopn}

\begin{document}

\maketitle

\begin{abstract}
 In this paper, we investigate an optimal control problem with terminal stochastic linear complementarity constraints (SLCC), and its discrete approximation using the relaxation, the sample average approximation (SAA) and the implicit Euler time-stepping scheme. We show the existence of feasible solutions and optimal solutions to the optimal control problem and its discrete approximation under the conditions that the expectation of the stochastic matrix in the SLCC is a Z-matrix or an adequate matrix. Moreover, we prove that the solution sequence  generated by the discrete approximation converges to a solution of the original optimal control problem with probability 1 as $\epsilon \downarrow  0$, $\nu\to \infty $ and $h\downarrow  0$,  where $\epsilon$ is the relaxation parameter, $\nu$ is the sample size and $h$ is the mesh size.  We also provide asymptotics of the SAA optimal value and error bounds of the time-stepping method. A numerical example is used to illustrate the existence of optimal solutions, the discretization scheme and error estimation.
\end{abstract}



\section{Introduction}

Let $\xi$ be a random variable defined in the probability space $(\Omega,\mathcal{F},\mathcal{P})$ with support set $\Xi:=\xi(\Omega)\subseteq\mathbb{R}^d$. Let $\|\cdot\|$ denote the Euclidean norm of a vector and a matrix.
Let $L^2(0,T)^n$ denote the Banach space of all quadratically Lebesgue integrable functions mapping from $(0,T)$ to $\mathbb{R}^n$, which is equipped with the norm
\[\|x\|_{L^2}:=\left(\int_0^T\|x(t)\|^2dt\right)^{\frac{1}{2}},\,\,\,\forall\,\,x\in L^2(0,T)^n.\]
Denote $H^1(0,T)^n$ the space of all functions $x\in L^2(0,T)^n$ whose components $x_1,\cdot\cdot\cdot,x_n:(0,T)\rightarrow\mathbb{R}$ possess weak derivatives such that the function $\dot{x}\in L^2(0,T)^n$. A suitable norm in $H^1(0,T)^n$ is defined by
\[\|x\|_{H^1}:=\left(\|x\|_{L^2}^2+\|\dot{x}\|_{L^2}^2 \right)^{\frac{1}{2}},\,\,\,\forall\,\,x\in H^1(0,T)^n.\]

In this paper, we aim to find an optimal  solution $(x,u)\in H^1(0,T)^n\times L^2(0,T)^m$ of the following optimal control problem with terminal stochastic linear complementarity constraints:
\begin{equation}\label{OCTCCS}
\begin{aligned}
& \min_{x,u}\,\,\Phi(x,u):=\mathbb{E}[F(x(T),\xi)]+\frac{1}{2}\|x-x_d\|_{L^2}^2+\frac{\delta}{2}\|u-u_d\|_{L^2}^2&\\
&\textrm{s.t.}
\left\{
\begin{aligned}
&\left.
\begin{aligned}
&\dot{x}(t)=Ax(t)+Bu(t), &\\
&Cx(t)+Du(t)-f(t)\leq 0,&\\
\end{aligned}\right\}a.e.\,\,t\in(0,T),&\\
&0\leq x(T)  \perp \mathbb{E}[M(\xi)x(T)+q(\xi)] \geq 0,&\\
&x(0)=x_0,\,\mathbb{E}[g(x(T),\xi)]\in K,
\end{aligned}
\right.
\end{aligned}
\end{equation}
where  $A\in \mathbb{R}^{n\times n}$, $B\in \mathbb{R}^{n\times m}$, $C\in \mathbb{R}^{l\times n}$, $D\in \mathbb{R}^{l\times m}$, $x_0\in\mathbb{R}^{n}$ and $f\in L^2(0,T)^l$ are given, $\delta>0$ is a scalar, $K\subseteq\mathbb{R}^\ell$ is a nonempty, closed and convex set. Here $x_d\in L^2(0,T)^n$ and $u_d\in L^2(0,T)^m$ are the given desired state and control, respectively. The mappings $F: \mathbb{R}^n\times \Xi\rightarrow \mathbb{R}$ and $g:\mathbb{R}^n\times \Xi\rightarrow \mathbb{R}^\ell$ are continuously differentiable over $\mathbb{R}^n$ for almost every $\xi\in\Xi$, and are measurable on $\Xi$ for each $z\in\mathbb{R}^n$. Two mappings $M: \Xi\rightarrow \mathbb{R}^{n\times n}$ and $q: \Xi\rightarrow \mathbb{R}^n$ are also measurable in $\xi\in\Xi$.

In \cite{FP2018}, Benita and Mehlita   studied an optimal control problem with terminal deterministic complementarity
constraints, which has many interesting practical applications in multi-agent control networks. They derived some stationarity conditions and presented constraint qualifications which ensure that these conditions hold at a local optimal solution of the optimal control problem under the assumption that the feasible set is nonempty.
However, sufficient conditions were not given for the existence of $x(T)$ such that the terminal deterministic complementarity constraints
\begin{eqnarray}
& 0\leq x(T)  \perp \bar{M}x(T)+\bar{q} \geq 0, \label{LCP}\\
&\bar{g}(x(T)) \in K,\label{LCPg}
\end{eqnarray}
hold, where $\bar{M}\in \mathbb{R}^{n\times n}$, $\bar{q}\in \mathbb{R}^{n}$ and $\bar{g}:\mathbb{R}^n\rightarrow \mathbb{R}^\ell$.
  Problem (\ref{OCTCCS}) extends the problem in \cite{FP2018} to terminal stochastic case in uncertain  environment.

Optimal control with  differential equations  and  complementarity constraints
provides a powerful modeling paradigm for many practical problems such as the optimal control of electrical networks with diodes and/or MOS transistors \cite{BB2003} and dynamic optimization of chemical processes \cite{AM2004}. It can also be derived from the KKT conditions of a bilevel optimal control if the lower level problem is convex and satisfies a constraint qualification  \cite{PG2020}. A series of works \cite{MP2013,ChenSICON,CD2019,GY2016,VB2020} are devoted to the study of optimal control problems with complementarity constraints.
It should be noted that these papers focus on deterministic problems, where the system coefficients including system parameters and boundary/initial conditions are perfectly known. On the other hand,
optimal control problems with stochastic differential equation constraints  under uncertain environment have been extensively studied \cite{PL2020,CP2014,CP2016}. These papers investigate theory and algorithms for optimal control when the parameters in the differential equations have noise and uncertainties. However,  there is very little research on  optimal control with terminal stochastic complementarity constraints.

The main contributions of this paper are summarized as follows.
We show the existence of feasible solutions  to the optimal control problem \eqref{OCTCCS} under the conditions that $\mathbb{E}[M(\xi)]$ is a Z-matrix or an adequate matrix, which gives reasonable conditions for  the existence of $x(T)$ such that (\ref{LCP})-(\ref{LCPg}) hold.
Moreover, we prove the  existence of feasible solutions and optimal solutions to the
 discrete approximation using the relaxation, the sample average approximation (SAA) and the implicit Euler time-stepping scheme under the same conditions.   In the convergence analysis, we prove that the solution sequence  generated by the discrete approximation converges to a solution of the original optimal control problem  with probability 1 (w.p.1) as $\epsilon\downarrow  0$, $\nu\to \infty $ and $h\downarrow  0$,  where $\epsilon$ is the relaxation parameter,  $\nu$ is the sample size and $h$ is the mesh size.  We also provide asymptotics of the SAA optimal value and error bounds of the time-stepping method. These results extend the approximation error of the Euler time-stepping method of an optimal control problem with convex terminal constraints to nonconvex terminal stochastic complementarity constraints.

The paper is organised as follows: Section \ref{se:Existence} deals with the existence of feasible solutions of problem \eqref{OCTCCS}. Section \ref{se:SAA} studies the existence of feasible solutions of the relaxation and  the SAA of \eqref{OCTCCS} and the convergence to the original problem \eqref{OCTCCS} as the relaxation parameter goes to zero and the sample size approaches to infinity. In Section \ref{se:time}, we study the convergence of the time-stepping scheme and show the convergence properties of the discrete method using the SAA and the implicit Euler time-stepping scheme. A numerical example is given in Section \ref{Example} to illustrate the theoretical results obtained in this paper.     Final conclusion remarks are presented in Section \ref{se:conclusions}.

\subsection{Notation and assumptions}

Throughout this paper we use the following notation. For a matrix $A\in \mathbb{R}^{m\times n}$, $A_{i\bullet}$ and $A_{\bullet j}$ denote its $i$th row and $j$th column, respectively,  $A^\top$ denotes its transpose matrix, and $A^\dag$ is its pseudoinverse matrix. If $A$ possesses full row rank $m$, we have $A^\dag=A^\top (AA^\top)^{-1}$.
Let $I$ denote the identity matrix with a certain dimension.

For sets ${S}_1,S_2\subseteq \mathbb{R}^n$, we denote the distance from $v\in \mathbb{R}^n$ to $S_1$ and the deviation of the set $S_1$ from the set $S_2$ by
$\textrm{dist}(v,S_1)=\inf_{v'\in S_1}\|v-v'\|,$
 and $\mathbb{D}(S_1,S_2)=\sup_{v\in S_1}\textrm{dist}(v,S_2)$, respectively. For sets ${S}_1,S_2\subseteq H^1(0,T)^n \times L^2(0,T)^m$, we denote the distance from $(v_1,v_2)\in H^1(0,T)^n \times L^2(0,T)^m$ to $S_1$ by
$\textrm{dist}((v_1,v_2),S_1)=\inf_{(v'_1,v'_2)\in S_1}(\|v_1-v'_1\|_{H^1}+\|v_2-v'_2\|_{L^2}),$
and the deviation of the set $S_1$ from the set $S_2$ by $\mathbb{D}(S_1,S_2)=\sup_{(v_1,v_2)\in S_1}\textrm{dist}((v_1,v_2),S_2)$. Let $\mathcal{B}(v,\varepsilon)=\{w: \|w-v\|\le\varepsilon\}$ be the closed ball centered at $v$ with the radius of $\varepsilon$.
Let int$S$ denote the interior of a set $S$. Let $[N]=\{1,2,\ldots, N\}.$

We give two basic assumptions  to ensure that problem \eqref{OCTCCS} is well defined.

\begin{assumption}\label{assumption3}
 There exist four nonnegative measurable functions $\kappa_i(\xi)$ with $\mathbb{E}[\kappa_i(\xi)]<\infty$ $(i=1,2,3,4)$ such that
 for any $z_1$, $z_2\in \mathbb{R}^n$,
\[|F(z_1,\xi)-F(z_2,\xi)|\leq \kappa_1(\xi)\|z_1-z_2\|,\,\,\|g(z_1,\xi)\|\leq \kappa_2(\xi)\|z_1\|,\,\,a.e.\,\,\xi\in\Xi,\]
and
  \[\|M(\xi)\|\leq \kappa_3(\xi)\,\,\,\textrm{and}\,\,\,\|q(\xi)\|\leq \kappa_4(\xi),\,\,\,\forall\,\, \xi\in\Xi.\]
\end{assumption}


\begin{assumption}\label{assumption-Y}
The matrix $D\in \mathbb{R}^{l\times m}$ is full row rank with $l<m$ and the matrix
\[\mathcal{R}:=[BY\,\,\,\, (A-BD^\dag C)BY\,\,\,(A-BD^\dag C)^2BY\,\,\, \cdot\cdot\cdot\,\, \,(A-BD^\dag C)^{n-1}BY]\in \mathbb{R}^{n\times n(m-l)}\]
is also full row rank, where $Y\in\mathbb{R}^{m\times (m-l)}$ is a matrix with full column rank $m-l$ such that $DY=0$.
\end{assumption}

\section{Existence of optimal solutions of problem \eqref{OCTCCS}}\label{se:Existence}

 In this section, we first investigate the feasibility of problem \eqref{OCTCCS}.
We call $(x,u)\in H^1(0,T)^n\times L^2(0,T)^m$ a feasible solution  of  \eqref{OCTCCS} if it satisfies the constraints in  \eqref{OCTCCS}.

For an index set $J \subseteq\{1,2,\cdot\cdot\cdot,n\}$, let $|J|$ denote its cardinality and $J^c$ denote its complementarity set.
We denote by $q_J \in \mathbb{R}^{|J|}$ the subvector formed from a vector
$q\in \mathbb{R}^n$ by picking the entries indexed by $J$ and denote by $M_{J_1,J_2}\in \mathbb{R}^{|J_1|\times |J_2|}$ the submatrix formed from a matrix $M\in \mathbb{R}^{n\times n}$ by picking the rows indexed
by $J_1$ and columns indexed by $J_2$.
  Let
${\cal J}=\{J \, :\, \mathbb{E}[M_{J,J}(\xi)] $ is nonsingular$\}$ and
$$\beta=\left\{\begin{array}{ll}
1 & {\rm if}  \,  {\cal J} =\emptyset,\\
\max\{\|(\mathbb{E}[M_{J,J}(\xi)])^{-1}\|_1\,\, | \,\,   J\in {\cal J} \} & {\rm otherwise.}
\end{array}
\right.
$$

A square matrix is said to be a P-matrix if all its principal minors are positive. A square matrix is said to be a Z-matrix if its off-diagonal entries are non-positive.
A matrix $\mathbb{E}[M(\xi)]\in \mathbb{R}^{n\times n} $ is called column adequate if for each $z\in \mathbb{R}^n$, $z_i(\mathbb{E}[M(\xi)]z)_i\leq 0$ for all $i=1,2,\cdot\cdot\cdot,n$ implies $\mathbb{E}[M(\xi)]z=0$. The matrix $\mathbb{E}[M(\xi)]$ is row adequate if $\mathbb{E}[M(\xi)]^\top $ is column adequate and it is adequate if it is both column and row adequate \cite{LCP1992}. It is known that a P-matrix is  adequate and a symmetric positive semi-definite matrix is also adequate \cite[Theorem 3.1.7, Theorem 3.4.4]{LCP1992}. However, an adequate matrix may neither be a P-matrix nor a positive semi-definite matrix \cite{LCP1992}.

\begin{theorem}\label{feasible-set}
Let Assumption \ref{assumption3} and Assumption \ref{assumption-Y} hold. Suppose the following three conditions hold: (i) $\mathcal{B}(0, \beta\mathbb{E}[\kappa_2(\xi)]\|\mathbb{E}[q(\xi)]\|_1) \subseteq K$,
(ii) the set
\[\mathcal{V}:=\{v\in \mathbb{R}^n\,|\, \mathbb{E}[M(\xi)v+q(\xi)]\geq 0,\,v\geq0\}\] is nonempty, and (iii) $\mathbb{E}[M(\xi)]$ is an adequate matrix or a Z-matrix.
 Then problem \eqref{OCTCCS} has a feasible solution $(x,u)\in H^1(0,T)^n\times L^2(0,T)^m$. Moreover, the optimal control problem \eqref{OCTCCS} admits an optimal solution if $\mathbb{E}[F(\cdot,\xi)]$ is bounded from below.
\end{theorem}
\begin{proof}

Following Lemma 7.2 in \cite{FP2018}, Assumption \ref{assumption-Y} implies that the following system
\begin{equation}\label{ODE-4}
\left\{
\begin{aligned}
& \dot{x}(t)=Ax(t)+Bu(t),&\\
& Cx(t)+Du(t)=p(t),&\\
& x(0)=x_0, \,\,\,x(T)=b,&\\
\end{aligned}
\right.\,\,a.e.\,\,t\in(0,T),
\end{equation}
admits a solution $(x,u)\in H^1(0,T)^n\times L^2(0,T)^m$ for any $p\in L^2(0,T)^l$ and $b\in \mathbb{R}^n$.
If we set $p(t)=f(t)+\tilde{p}(t)$ in \eqref{ODE-4} for arbitrary $\tilde{p}\in L^2(0,T)^l$ with $\tilde{p}(t)\leq 0$ and $f(t)$ in \eqref{OCTCCS}, then  the following problem
\begin{equation}\label{ODE-inequality}
\left\{
\begin{aligned}
& \dot{x}(t)=Ax(t)+Bu(t),&\\
& Cx(t)+Du(t)-f(t)\leq 0,&\\
& x(0)=x_0, \,\,\,x(T)=b,&\\
\end{aligned}
\right.\,\,a.e.\,\,t\in(0,T),
\end{equation}
has a solution $(x,u)\in H^1(0,T)^n\times L^2(0,T)^m$ for any $b\in\mathbb{R}^n$.

Now we show the solution set of the following stochastic constrained LCP is nonempty,
\begin{equation}\label{mixedCP}
\left\{
\begin{aligned}
& \min \{x(T), \, \mathbb{E}[M(\xi)x(T)+q(\xi)]\}=0,&\\
&\mathbb{E}[g(x(T),\xi)]\in K.
\end{aligned}
\right.
\end{equation}

Following Corollary 3.3.5 and  Theorem 3.11.6 in \cite{LCP1992}, the LCP in \eqref{mixedCP} has a solution from the assumption that the set $\mathcal{V}$ is nonempty and $\mathbb{E}[M(\xi)]$ is adequate or a
Z-matrix. Let $x^*(T)$ be a sparse solution of
the LCP in (\ref{mixedCP}), which is defined as
$$\begin{aligned}
x^*(T) \in & \quad {\rm argmin}\,  \|x(T)\|_0\\
& \quad {\rm s.t. } \quad \quad  0\leq x(T)  \perp \mathbb{E}[M(\xi)x(T)+q(\xi)] \geq 0,
\end{aligned}
$$
where $\|x(T)\|_0=$ the number of nonzero components of $x(T)$.
If there is no $J$ such that $\mathbb{E}[M_{J,J}(\xi)]$ is nonsingular, that is, ${\cal J}=\emptyset$,  then by Theorem 2.2 in \cite{ChenXiang} and Theorem 3.1 in \cite{chen-xiang-2016}, $\|x^*(T)\|_0=\|x^*(T)\|_1=0.$ Hence, we have
\begin{equation}\label{xT}
\|x^*(T)\|\le \beta\|\mathbb{E}[q(\xi)]\|_1.
\end{equation}
If there is $J$ such that $x^*(T)_J>0$ and $x^*(T)_{J^c}=0$,
from Theorem 2.2 in \cite{ChenXiang} and Theorem 3.1 in \cite{chen-xiang-2016}, we know that $\mathbb{E}[M_{J,J}(\xi)]$ is nonsingular and
\[x^*(T)=-(I-\Lambda+\Lambda\mathbb{E}[M(\xi)])^{-1}\Lambda\mathbb{E}[q(\xi)],\]
where $\Lambda$ is a diagonal matrix with $\Lambda_{i,i}=1,$ if $i\in J$ and $\Lambda_{i,i}=0,$ if $i\in J^c$.
Moreover,   from
$$\|(I-\Lambda+\Lambda\mathbb{E}[M(\xi)])^{-1}\Lambda\|
\le \max\{\|(\mathbb{E}[M_{J,J}(\xi)])^{-1}\|_1 \, |\,   J\in {\cal J} \},$$
we obtain (\ref{xT}) for ${\cal J}\neq \emptyset$.

Therefore, from Assumption \ref{assumption3} and assumption (i) of this theorem,  we have
\begin{equation*}
\begin{aligned}
&\|\mathbb{E}[g(x^*(T),\xi)]\|\leq \mathbb{E}[\kappa_2(\xi)]\|x^*(T)\|\leq\mathbb{E}[\kappa_2(\xi)] \|x^*(T)\|_1\le \beta\mathbb{E}[\kappa_2(\xi)]\|\mathbb{E}[q(\xi)]\|_1,
\end{aligned}
\end{equation*}
which implies that $\mathbb{E}[g(x^*(T),\xi)]\in  K$.  Hence the solution set of \eqref{mixedCP} is nonempty.

Similar to the proof of Theorem 5.1 in \cite{FP2018}, we can derive the existence of optimal solutions to problem \eqref{OCTCCS} under the assumption that $\mathbb{E}[F(\cdot,\xi)]$ is bounded from below.
\end{proof}

\begin{remark}\label{RE-1}
The constrained LCP \eqref{mixedCP} may have  multiple solutions or may not have a solution. If $\mathbb{E}[M(\xi)]$ is a P-matrix, then for any $\mathbb{E}[q(\xi)]$, the LCP in \eqref{mixedCP} has a unique solution $x(T)$. In such case, if $\mathbb{E}[g(x(T),\xi)]\in K$, then \eqref{mixedCP} has a unique solution, otherwise \eqref{mixedCP} does not have a solution. If $\mathbb{E}[M(\xi)]$ is a Z-matrix or an adequate matrix, the LCP in \eqref{mixedCP} may have multiple solutions, while some solutions can be bounded by $\beta \|\mathbb{E}[q(\xi)]\|_1$. When $\mathcal{B}(0, \beta\mathbb{E}[\kappa_2(\xi)]\|\mathbb{E}[q(\xi)]\|_1)$ $\subseteq K$, some solutions of the LCP satisfy $\mathbb{E}[g(x(T),\xi)]\in K$
and thus the constrained LCP \eqref{mixedCP} is solvable. See the example
in Section \ref{Example}.
\end{remark}

\begin{remark}\label{RE-2}
Assumption \ref{assumption-Y} is also used in \cite{FP2018} for  the case $l<m$, which allows more freedom for the system controls. If $l=m$ and $D$ is invertible, we can  write $Cx(t)+Du(t)-f(t)=-v(t)$ with $v(t)\geq 0$ for a.e. $t\in[0,T]$, where $v\in L^2(0,T)^l$. Then the solvability of \eqref{ODE-inequality} becomes to find a solution pair $(x,v)\in H^1(0,T)^n \times L^2(0,T)^l$ with $v(t)\geq 0$ satisfying
\begin{equation}\label{ODE-positive-control-new}
\left\{
\begin{aligned}
& \dot{x}(t)=(A-BD^{-1} C)x(t)+BD^{-1} f(t) -BD^{-1} v(t),&\\
& x(0)=x_0, \,\,\,x(T)=b,&\\
\end{aligned}
\right.\,\,a.e.\,\,t\in(0,T).
\end{equation}
It then requires the concept of positive controllability \cite{Brammer-1972,hiroshi-tetsuro-2007}.
 Therefore, the conditions to ensure the solution set of \eqref{ODE-positive-control-new} is nonempty for any $b\in \mathbb{R}^n$ are
(i) the matrix
\[[BD^{-1}\,\,\,\, (A-BD^{-1} C)BD^{-1}\,\,\,(A-BD^{-1} C)^2BD^{-1}\,\,\, \cdot\cdot\cdot\,\, \,(A-BD^{-1} C)^{n-1}BD^{-1}]\]
possesses full row rank $n$, and (ii) there is no real eigenvector $\textbf{w}\in \mathbb{R}^n$ of $(A-BD^{-1} C)^\top$ such that $\textbf{w}^\top BD^{-1}\textbf{v} \geq 0$ for any $\textbf{v}\in \mathbb{R}^m_+$. Then there is a finite time $T_0$ such that the solution set of \eqref{ODE-inequality} is nonempty for any $b\in\mathbb{R}^n$ and $T\geq T_0$.  Hence we can replace Assumption \ref{assumption-Y} in Theorem \ref{feasible-set} by these two conditions for the case  that $l=m$ and $D$ is invertible.
\end{remark}

\section{Relaxation and sample average approximation (SAA) }\label{se:SAA}

In this section, we apply the relaxation and the SAA approach  to solve \eqref{OCTCCS}.
We consider an independent identically distributed (i.i.d) sample of $\xi(\omega)$, which is
denoted by $\{\xi_1,\cdot\cdot\cdot,\xi_\nu\}$, and use the following relaxation and SAA problem to
approximate problem \eqref{OCTCCS}:
\begin{equation}\label{OCTCCS-sample}
\begin{aligned}
& \min_{x,u}\,\,\Phi^{\nu}(x,u):=\frac{1}{\nu}\sum_{\ell=1}^\nu  F(x(T),\xi_\ell)+\frac{1}{2}\|x-x_d\|_{L^2}^2+\frac{\delta}{2}\|u-u_d\|_{L^2}^2&\\
&\textrm{s.t.}
\left\{
\begin{aligned}
&\left.
\begin{aligned}
&\dot{x}(t)=A x(t)+Bu(t), &\\
& C x(t)+Du(t)-f(t)\leq 0,
\end{aligned}\right\}a.e.\,\,t\in(0,T),&\\
&\left\|\min \left\{x(T), \frac{1}{\nu}\sum_{\ell=1}^\nu[M(\xi_\ell)x(T)+q(\xi_\ell)]\right\}\right\|\le \epsilon, &\\
&x(0)=x_0,\,\,\, \frac{1}{\nu}\sum_{\ell=1}^\nu g(x(T),\xi_\ell)\in
 K^\epsilon :=\left\{ z \, | \, {\rm dist}(z,K) \le \epsilon\right\},
\end{aligned}
\right.
\end{aligned}
\end{equation}
where $\epsilon>0$ is a sufficiently small number.

\subsection{Convergence of the relaxation and SAA}
In this subsection, we show the existence of a solution of problem
(\ref{OCTCCS-sample}), and its convergence as $\epsilon \downarrow 0$ and $\nu \to \infty$.

\begin{theorem}\label{feasible-set-1}
Suppose that the conditions in Theorem \ref{feasible-set} hold.
Then for any $\epsilon>0$, the SAA problem \eqref{OCTCCS-sample} has an optimal solution $(x^{\epsilon,\nu},u^{\epsilon,\nu})\in H^1(0,T)^n\times L^2(0,T)^m$ for sufficiently large $\nu$ w.p.1.
\end{theorem}
\begin{proof}
Since the solution set of the linear control system \eqref{ODE-inequality} is nonempty for any $b\in \mathbb{R}^n$, for the existence of a feasible solution to the SAA problem \eqref{OCTCCS-sample}, it suffices to show the solution set of the following system
\begin{equation}\label{mixedCP-SAA}
\left\{
\begin{aligned}
&\left\|\min \left\{x(T), \frac{1}{\nu}\sum_{\ell=1}^\nu[M(\xi_\ell)x(T)+q(\xi_\ell)]\right\}\right\|\le \epsilon, &\\
& \frac{1}{\nu}\sum_{\ell=1}^\nu g(x(T),\xi_\ell)\in
 K^\epsilon
\end{aligned}
\right.
\end{equation}
is nonempty w.p.1 for any positive number $\epsilon$.

Let $x^*(T)$ be a sparse solution of
the LCP in (\ref{mixedCP}).
From Theorem \ref{feasible-set}, we know that $x^*(T)$ satisfies (\ref{mixedCP}). By the Law of Large Number,
for sufficiently lager $\nu$,  $x^*(T)$ is a solution of (\ref{mixedCP-SAA}).
It concludes that the solution set of the system \eqref{mixedCP-SAA} is nonempty w.p.1 for sufficiently  large $\nu$.

Since $\mathbb{E}[ F(\cdot,\xi)]$ is bounded from below, we can also obtain that $\frac{1}{\nu}\sum_{\ell=1}^\nu F(\cdot,\xi_\ell)$ is bounded from below with sufficiently large $\nu$. The existence of optimal solutions to problem \eqref{OCTCCS-sample} is  similar to the proof of Theorem \ref{feasible-set}.
\end{proof}

Since $F(\cdot,\xi)$ is a smoothing function for a.e. $\xi\in\Xi$,
following the proof of Lemma 3.5 in \cite{Chen-Wets-Zhang}, we can have the following lemma.
\begin{lemma}\label{epi-con-objective}
Let $\mathcal{C}_1\times \mathcal{C}_2$ denote a compact subset of $H^1(0,T)^n\times L^2(0,T)^m$. It holds that $\Phi^{\nu}$ epiconverges to $\Phi$ w.p.1
over $\mathcal{C}_1\times \mathcal{C}_2$ as $\nu\rightarrow\infty$.
\end{lemma}

Let $\mathcal{Z}^{\epsilon,\nu}$ and $\mathcal{Z}$ denote the solution sets of \eqref{mixedCP-SAA} and \eqref{mixedCP}, respectively.
Let  $\mathcal{S}^{\epsilon,\nu}$ and $\mathcal{S}$ be the feasible solution sets, and  $\hat{\mathcal{S}}^{\epsilon,\nu}$ and $\hat{\mathcal{S}}$ be optimal solution sets of \eqref{OCTCCS-sample} and \eqref{OCTCCS}, respectively.

\begin{theorem}\label{lemma-set-epi}
 Suppose that the conditions of Theorem \ref{feasible-set} hold, then it holds that $\lim_{\epsilon\downarrow0}\lim_{\nu\rightarrow\infty}\mathbb{D}(\mathcal{Z}^{\epsilon,\nu},\mathcal{Z})=0$ w.p.1, $ \lim_{\epsilon\downarrow0}\lim_{\nu\rightarrow\infty}\mathbb{D}(\mathcal{S}^{\epsilon,\nu},\mathcal{S})=0$ w.p.1.
and $\lim_{\epsilon\downarrow0}\lim_{\nu\rightarrow\infty}\mathbb{D}(\hat{\mathcal{S}}^{\epsilon,\nu},\hat{\mathcal{S}})=0$ w.p.1.
\end{theorem}
\begin{proof}
From Theorem \ref{feasible-set} and Theorem \ref{feasible-set-1}, $\mathcal{Z}$ and
$\mathcal{Z}^{\epsilon,\nu}$ with any $\epsilon>0$ and large enough $\nu$  are nonempty.
By (\ref{xT}) and  $\mathcal{B}(0, \beta\mathbb{E}[\kappa_2(\xi)]\|\mathbb{E}[q(\xi)]\|_1) \subseteq  K$ in Theorem \ref{feasible-set},
 we know that there is a compact set $\mathcal{X}$ such that $\mathcal{Z}\subseteq \mathcal{X}$ and $\mathcal{Z}^{\epsilon,\nu}\subseteq \mathcal{X}$ for sufficiently small $\epsilon$ and sufficiently large $\nu$.
Let
 $$
\phi(x(T)):=\min \{x(T),  \mathbb{E}[M(\xi)x(T)+q(\xi)]\}
\quad {\rm and} \quad \psi(x(T)):=\mathbb{E}[g(x(T),\xi)].
$$
For $x(T) \in\mathcal{Z}$, $\phi(x(T))=0$ and $\psi(x(T))\in K.$
From (\ref{mixedCP-SAA}), for $x(T) \in\mathcal{Z}^{\epsilon,\nu}$, there are $v^\nu \in \mathbb{R}^n, w^\nu\in \mathbb{R}^\ell$ with $\|v^\nu \|\le \epsilon$ and $ \|w^\nu\|\le \epsilon$ such that
 $$
\begin{aligned}
&\phi^\nu_\epsilon (x(T)):=\min \left\{x(T),  \frac{1}{\nu}\sum_{\ell=1}^\nu[M(\xi_\ell)x(T)+q(\xi_\ell)]\right\} +v^\nu=0, &\\
& \psi^\nu_\epsilon (x(T)):=\frac{1}{\nu}\sum_{\ell=1}^\nu g(x(T),\xi_\ell) +w^\nu \in
 K.
\end{aligned}
$$
Since $\phi$ and $\psi$ are continuous, and $M(\cdot), q(\cdot)$ and $ g(x(T),\cdot)$
satisfy Assumption \ref{assumption3}, we have
$\phi^\nu_\epsilon$ and $\psi^\nu_\epsilon$ converge to $\phi$ and $\psi$ uniformly, respectively on the compact set $\mathcal{X}$, that is,
$$\lim_{\epsilon \downarrow 0}\lim_{\nu \to \infty} \max_{x(T)\in \mathcal{X}}\|\phi^\nu_\epsilon(x(T))-\phi(x(T))\|=0,$$
and
$$ \lim_{\epsilon \downarrow 0}\lim_{\nu \to \infty} \max_{x(T)\in \mathcal{X}}\|\psi^\nu_\epsilon(x(T))-\psi(x(T))\|=0.$$
Therefore, following Theorem 5.12 in \cite{Shapiro2009}, $\lim_{\epsilon\downarrow0}\lim_{\nu\rightarrow\infty}\mathbb{D}(\mathcal{Z}^{\epsilon,\nu},\mathcal{Z})=0$ w.p.1.

Now we show $\lim_{\epsilon\downarrow0}\lim_{\nu\rightarrow\infty}\mathbb{D}(\mathcal{S}^{\epsilon,\nu},\mathcal{S})=0$ holds w.p.1. Note that $\mathcal{S}^{\epsilon,\nu}$ and $\mathcal{S}$ are two nonempty closed sets.
Obviously, two nonempty closed sets $\mathcal{S}$ and $\mathcal{S}^{\epsilon,\nu}$ are the solution sets of problem \eqref{ODE-inequality} with terminal sets $\mathcal{Z}$ and $\mathcal{Z}^{\epsilon,\nu}$, respectively.
For any $p\in L^2(0,T)^l$, the pair $(\|x\|_{H^1}, \|u\|_{L^2})$, where $(x,u)$ is a solution of problem \eqref{ODE-4}, is uniquely defined by the terminal point $x(T)$. In addition, it is clear that a solution $(x,u)$ of problem \eqref{ODE-4}  is continuous with respect to the terminal point $x(T)$. Hence, for any $(x^{\epsilon,\nu},u^{\epsilon,\nu})\in \mathcal{S}^{\epsilon,\nu}$ and $(x,u)\in \mathcal{S}$, we have $(x^{\epsilon,\nu},u^{\epsilon,\nu})\rightarrow (x,u)$ w.p.1 in the norm $\|\cdot\|_{H^1}\times \|\cdot\|_{L^2}$ when $x^{\epsilon,\nu}(T)\rightarrow x(T)$ w.p.1 as $\epsilon\downarrow0$ and $\nu\rightarrow\infty$. It then concludes $\lim_{\epsilon\downarrow0}\lim_{\nu\rightarrow\infty}\mathbb{D}(\mathcal{S}^{\epsilon,\nu},\mathcal{S})=0$ w.p.1.

It is clear that, $\arg\min \Phi=\hat{\mathcal{S}}\subseteq \mathcal{S}$ and $\arg\min \Phi^{\nu}=\hat{\mathcal{S}}^{\epsilon,\nu} \subseteq \mathcal{S}^{\epsilon,\nu}$ for sufficiently small $\epsilon$ and sufficiently large $\nu$. By $\lim_{\epsilon\downarrow0}\lim_{\nu\rightarrow\infty}\mathbb{D}(\mathcal{S}^{\epsilon,\nu},\mathcal{S})=0$ w.p.1, we have, for $(\hat{x}^{\epsilon,\nu},\hat{u}^{\epsilon,\nu}) \in \hat{\mathcal{S}}^{\epsilon,\nu}$, there is $(\hat{x},\hat{u})\in \mathcal{S}$ such that $(\hat{x}^{\epsilon,\nu},\hat{u}^{\epsilon,\nu})\rightarrow (\hat{x},\hat{u})$ in the norm $\|\cdot\|_{H^1}\times \|\cdot\|_{L^2}$ as $\epsilon\downarrow0$ and $\nu\rightarrow\infty$. In addition, according to Theorem 2.5 in \cite{Wets1981}, we obtain $(\hat{x},\hat{u})\in \hat{\mathcal{S}}$ by $\Phi^{\nu}\xrightarrow{epi} \Phi$ w.p.1, which implies $\lim_{\epsilon\downarrow0}\lim_{\nu\rightarrow\infty}\mathbb{D}(\hat{\mathcal{S}}^{\epsilon,\nu},\hat{\mathcal{S}})=0$ w.p.1.
\end{proof}

\subsection{Asymptotics of the SAA optimal value}\label{se:optimal-value}

We introduce the relaxation of problem \eqref{OCTCCS} with a parameter $\epsilon>0$ as follows
\begin{equation}\label{OCTCCS-epsilon}
\begin{aligned}
& \min_{x,u}\,\,\Phi(x,u)&\\
&\textrm{s.t.}
\left\{
\begin{aligned}
&\left.
\begin{aligned}
&\dot{x}(t)=Ax(t)+Bu(t), &\\
&Cx(t)+Du(t)-f(t)\leq 0,&\\
\end{aligned}\right\}a.e.\,\,t\in(0,T),&\\
&\| \min\{ x(T), \mathbb{E}[M(\xi)x(T)+q(\xi)]\} \| \leq \epsilon,&\\
&x(0)=x_0,\,\mathbb{E}[g(x(T),\xi)]\in K^\epsilon.
\end{aligned}
\right.
\end{aligned}
\end{equation}

Denote by $\mathcal{Z}^\epsilon$ the solution set of the terminal constraints of \eqref{OCTCCS-epsilon}. Denote by $\mathcal{S}^\epsilon$ and $\hat{\mathcal{S}}^\epsilon$ the feasible solution set and optimal solution set of \eqref{OCTCCS-epsilon}, respectively.
Recall that $\mathcal{Z}$ is the solution set of \eqref{mixedCP}, and $\mathcal{S}$ and $\hat{\mathcal{S}}$ are the feasible solution set and optimal solution set of \eqref{OCTCCS}, respectively.
 It is clear that $\mathcal{Z}\subseteq \mathcal{Z}^\epsilon$ and $\mathcal{S}\subseteq \mathcal{S}^\epsilon$, which mean that $\Phi(\hat{x}^\epsilon,\hat{u}^\epsilon)\leq \Phi(\hat{x},\hat{u})$ for any $(\hat{x}^\epsilon,\hat{u}^\epsilon)\in \hat{\mathcal{S}}^\epsilon $ and $(\hat{x},\hat{u})\in \hat{\mathcal{S}} $. Therefore, $\mathcal{Z}^\epsilon$, $\mathcal{S}^\epsilon$ and $\hat{\mathcal{S}}^\epsilon$ are nonempty since $\mathcal{Z}$ and $\mathcal{S}$ are nonempty.

According to Theorem \ref{feasible-set} and Remark \ref{RE-1}, we also conclude that
$\mathcal{Z}^\epsilon$ and $\hat{\mathcal{S}}^\epsilon$ are compact. It can also be derived that $\lim_{\epsilon\downarrow0}\mathbb{D}(\mathcal{Z}^\epsilon,\mathcal{Z})=0$, $\lim_{\epsilon\downarrow0}\mathbb{D}(\mathcal{S}^\epsilon,\mathcal{S})=0$ and $\lim_{\epsilon\downarrow0}\mathbb{D}(\hat{\mathcal{S}}^\epsilon,\hat{\mathcal{S}})=0$. It is clear that \eqref{OCTCCS-sample} is the corresponding SAA problem of \eqref{OCTCCS-epsilon}. By Theorem \ref{lemma-set-epi}, we conclude that $\lim_{\nu\rightarrow\infty}\mathbb{D}({\mathcal{Z}}^{\epsilon,\nu},{\mathcal{Z}}^{\epsilon})=0$, $\lim_{\nu\rightarrow\infty}\mathbb{D}({\mathcal{S}}^{\epsilon,\nu},{\mathcal{S}}^{\epsilon})=0$ and $\lim_{\nu\rightarrow\infty}\mathbb{D}(\hat{\mathcal{S}}^{\epsilon,\nu},\hat{\mathcal{S}}^{\epsilon})=0$.

In the rest of this section, we study the asymptotics of  optimal value of the SAA problem \eqref{OCTCCS-sample} for a fixed $\epsilon>0$.

Since $\min\{x(T), \mathbb{E}[M(\xi)x(T)+q(\xi)]\}=0$
and $\mathbb{E}[g(x(T),\xi)]\in K$ for any $x(T)\in \mathcal{Z}$,
we have $\mathcal{Z}\subseteq \textrm{int} \mathcal{Z}^\epsilon$,
 which means that $\textrm{int} \mathcal{Z}^\epsilon \neq \emptyset$.
 Let
 $$\hat{\mathcal{Z}}=\{x(T):\,(x,u)\in \hat{\mathcal{S}}\}\quad {\rm  and }\quad \hat{\mathcal{Z}}^\epsilon=\{x(T):\,(x,u)\in\hat{\mathcal{S}}^\epsilon\}.$$ Obviously, we have $\hat{\mathcal{Z}}\subseteq \textrm{int} \mathcal{Z}^\epsilon$ and $\lim_{\epsilon\downarrow0}\mathbb{D}(\hat{\mathcal{Z}}^{\epsilon},\hat{\mathcal{Z}})=0$.

 We give the following assumptions.
\begin{assumption}\label{assumption-interior-point}
 The set  $\hat{\mathcal{Z}}$ is a singleton.
\end{assumption}

\begin{assumption}\label{assumption-square}
\begin{itemize}
  \item[(i)]  There exists a nonnegative measurable function $\kappa_1(\xi)$ with $\mathbb{E}[\kappa_1^2(\xi)]<\infty$ such that
 for any $z_1$, $z_2\in \mathbb{R}^n$ and $\xi\in\Xi$,
\[|F(z_1,\xi)-F(z_2,\xi)|\leq \kappa_1(\xi)\|z_1-z_2\|,\]
and $\mathbb{E}[F^2(z,\xi)]<\infty$ for any $z\in \mathbb{R}^n$.
  \item[(ii)] The function $\mathbb{E}[F(\cdot,\xi)]$ is a strongly convex function, that is, there is a constant $\mu>0$ such that, for any $z_1$, $z_2\in \mathbb{R}^n$ and $\tau\in(0,1)$,
\begin{eqnarray*}
\hspace{-0.3in}\mathbb{E}[F((1-\tau)z_1+\tau z_2,\xi)]\leq (1-\tau)\mathbb{E}[F(z_1,\xi)]+\tau \mathbb{E}[F(z_2,\xi)]-\frac{\mu \tau(1-\tau)}{2}\|z_1-z_2\|^2.
\end{eqnarray*}
\end{itemize}
\end{assumption}
\begin{theorem}\label{error-optimal-values-saa-epsilon}
Suppose that the conditions of Theorem \ref{feasible-set}, Assumption  \ref{assumption-interior-point} and Assumption \ref{assumption-square} hold. Let $(\hat{x}^\epsilon,\hat{u}^\epsilon)$ and $(\hat{x}^{\epsilon,\nu},\hat{u}^{\epsilon,\nu})$ be optimal solutions of \eqref{OCTCCS-epsilon} and \eqref{OCTCCS-sample}, respectively.
Then for sufficiently small $\epsilon$ and sufficiently large $\nu$, we have
 \[\sqrt{\nu}(\Phi^\nu(\hat{x}^{\epsilon,\nu},\hat{u}^{\epsilon,\nu})
-\Phi(\hat{x}^\epsilon,\hat{u}^\epsilon))\quad  \xrightarrow{D} \quad \mathcal{N}(0,\sigma^2(\hat{x}^\epsilon(T))),\]
where `` $ \xrightarrow{D}$ '' denotes convergence in distribution and $\mathcal{N}(0,\sigma^2(\hat{x}^\epsilon(T)))$ denotes the norm distribution with mean 0 and variance $\sigma^2(\hat{x}^\epsilon(T)):=\mathbb{V}\textrm{ar}[F(\hat{x}^\epsilon(T),\xi)]$.
\end{theorem}
\begin{proof}
Since $\hat{\mathcal{Z}}$ is a singleton, $\hat{\mathcal{Z}}\subseteq \textrm{int} \mathcal{Z}^\epsilon$ and $\lim_{\epsilon\downarrow0}\mathbb{D}(\hat{\mathcal{Z}}^{\epsilon},\hat{\mathcal{Z}})=0$, we have $\hat{\mathcal{Z}}^\epsilon\subseteq \textrm{int} \mathcal{Z}^\epsilon$  for sufficiently small $\epsilon$, which means that there is a convex set $\mathcal{Z}_{\mathcal{X}}$ such that $\hat{\mathcal{Z}}^{\epsilon}\subseteq \mathcal{Z}_{\mathcal{X}} \subseteq \mathcal{Z}^\epsilon$ for sufficiently small $\epsilon$.
We can also obtain that $\hat{\mathcal{Z}}^{\epsilon}$ is a singleton for sufficiently small $\epsilon$ under Assumption \ref{assumption-square}(ii). We argue it by contradiction. Suppose $(\hat{x}^\epsilon,\hat{u}^\epsilon)$ and $(\check{x}^\epsilon,\check{u}^\epsilon)$ are two optimal solutions of \eqref{OCTCCS-epsilon} with $\hat{x}^\epsilon(T)\neq\check{x}^\epsilon(T)$. Then $(x^\epsilon_\tau,u^\epsilon_\tau):=((1-\tau)\hat{x}^\epsilon+\tau \check{x}^\epsilon,(1-\tau)\hat{u}^\epsilon+\tau \check{u}^\epsilon)$ with $\tau\in(0,1)$ is also a feasible solution of \eqref{OCTCCS-epsilon}, since $x^\epsilon_\tau(T)\in \mathcal{Z}_{\mathcal{X}} \subseteq \mathcal{Z}^\epsilon$. Moreover,
\begin{eqnarray*}
& & \Phi(x^\epsilon_\tau,u^\epsilon_\tau)\leq (1-\tau)\Phi(\hat{x}^\epsilon,\hat{u}^\epsilon)+\tau\Phi(\check{x}^\epsilon,\check{u}^\epsilon)
-\frac{\mu \tau(1-\tau)}{2}\|\hat{x}^\epsilon(T)-\check{x}^\epsilon(T)\|^2,
\end{eqnarray*}
which means $\Phi(x^\epsilon_\tau,u^\epsilon_\tau)<\Phi(\hat{x}^\epsilon,\hat{u}^\epsilon)$ since $\Phi(\hat{x}^\epsilon,\hat{u}^\epsilon)=\Phi(\check{x}^\epsilon,\check{u}^\epsilon)$ and $\hat{x}^\epsilon(T)\neq\check{x}^\epsilon(T)$. It contradicts the assumption that $(\hat{x}^\epsilon,\hat{u}^\epsilon)$ is an optimal solution of \eqref{OCTCCS-epsilon}, and then we know that $\hat{\mathcal{Z}}^{\epsilon}$ is a singleton for sufficiently small $\epsilon$.

In the following argument, $\epsilon>0$ is a fixed number such that $\hat{\mathcal{Z}}^\epsilon$ is
singleton and $\hat{\mathcal{Z}}^\epsilon \subseteq \textrm{int}{\mathcal{Z}}^\epsilon$.
Denote $\hat{\mathcal{Z}}^{\epsilon,\nu}=\{x(T):\,(x,u)\in \hat{\mathcal{S}}^{\epsilon,\nu}\}$. We then obtain that $\lim_{\nu\rightarrow \infty}\mathbb{D}(\hat{\mathcal{Z}}^{\epsilon,\nu},\hat{\mathcal{Z}}^{\epsilon})=0$ and $\hat{\mathcal{Z}}^{\epsilon}\subseteq \textrm{int} \mathcal{Z}^{\epsilon,\nu}$ for sufficiently large $\nu$ according to $\lim_{\nu\rightarrow \infty}\mathbb{D}({\mathcal{Z}}^{\epsilon,\nu},{\mathcal{Z}}^{\epsilon})=0$.
Therefore, there is a $(\hat{x}^{\epsilon,\nu},\hat{u}^{\epsilon,\nu})\in \hat{\mathcal{S}}^{\epsilon,\nu}$ such that $\hat{x}^{\epsilon,\nu}(T)\in \textrm{int} \mathcal{Z}^{\epsilon,\nu}$ with sufficiently large $\nu$, which implies that,
for sufficiently large $\nu$, there is a compact set $\mathcal{X}$ such that $\hat{\mathcal{Z}}^{\epsilon}\subseteq \mathcal{X}\subseteq \mathcal{Z}^\epsilon$ and $\hat{x}^{\epsilon,\nu}(T)\in\mathcal{X}\subseteq \mathcal{Z}^{\epsilon,\nu}$.

 The solution $(x,u)$ of ODE \eqref{ODE-inequality} is continuous with respect to the state terminal value $x(T)$ and the pair $(\|x\|_{H^1},\|u\|_{L^2})$ is uniquely defined by $x(T)$. Therefore,
  there is a compact set $\mathfrak{X}$ such that $\hat{\mathcal{S}}^{\epsilon} \subseteq \mathfrak{X} \subseteq {\mathcal{S}}^{\epsilon}$ and $ \mathfrak{X} \subseteq {\mathcal{S}}^{\epsilon,\nu}$ with $\hat{\mathcal{S}}^{\epsilon,\nu}\cap \mathfrak{X} \neq \emptyset$.
 To derive the error of approximation for optimal value of \eqref{OCTCCS-sample} to that of \eqref{OCTCCS-epsilon}, it suffices to investigate the error approximation for optimal value of the following problem
 \begin{eqnarray}\label{epsilon-true}
 \min_{(x,u)\in \mathfrak{X}} \Phi(x,u)
 \end{eqnarray}
 and its SAA problem
 \begin{eqnarray}\label{epsilon-SAA}
 \min_{(x,u)\in \mathfrak{X}} \Phi^\nu(x,u),
 \end{eqnarray}
where $\Phi$ and $\Phi^\nu$ are defined in \eqref{OCTCCS-epsilon} and \eqref{OCTCCS-sample}, respectively. Clearly, $\mathfrak{X} \subseteq {\mathcal{S}}^{\epsilon}$ with $\hat{\mathcal{S}}^{\epsilon} \cap \mathfrak{X} \neq \emptyset$ and $ \mathfrak{X} \subseteq {\mathcal{S}}^{\epsilon,\nu}$ with $\hat{\mathcal{S}}^{\epsilon,\nu}\cap \mathfrak{X} \neq \emptyset$ mean that an optimal solution of  \eqref{epsilon-true}  is an optimal solution of \eqref{OCTCCS-epsilon}, and an optimal solution of  \eqref{epsilon-SAA} is also an optimal solution of \eqref{OCTCCS-sample}.
Therefore, according to Theorem 5.7 in \cite{Shapiro2009}, we can obtain that, under Assumption \ref{assumption-square},
\begin{eqnarray*}
&&\sqrt{\nu}(\Phi^\nu(\hat{x}^{\epsilon,\nu},\hat{u}^{\epsilon,\nu})
-\Phi(\hat{x}^\epsilon,\hat{u}^\epsilon))\\
&&=\sqrt{\nu}
(\frac{1}{\nu} \sum_{\ell=1}^\nu F(\hat{x}^{\epsilon,\nu}(T),\xi_\ell)
+\frac{1}{2}\|\hat{x}^{\epsilon,\nu}-x_d\|_{L^2}^2+
\frac{\delta}{2}\|\hat{u}^{\epsilon,\nu}-u_d\|_{L^2}^2\\
&& \quad \quad \quad - \mathbb{E}[F(\hat{x}^\epsilon(T),\xi)] - \frac{1}{2}\|\hat{x}^{\epsilon}-x_d\|_{L^2}^2-
\frac{\delta}{2}\|\hat{u}^{\epsilon}-u_d\|_{L^2}^2)\\
&&\xrightarrow{D} \inf_{(x,u)\in\hat{\mathcal{S}}^{\epsilon} } Y(x,u),
\end{eqnarray*}
where $Y(x,u)$ has a normal distribution with mean 0 and variance $\mathbb{V}\textrm{ar}[F(x(T),\xi)]$ with $(x,u)\in\hat{\mathcal{S}}^{\epsilon} $. Since $\hat{\mathcal{Z}}^\epsilon=\{\hat{x}^\epsilon(T)\}$ is a singleton,  $Y(x,u)$ for any $(x,u)\in\hat{\mathcal{S}}^{\epsilon}$ has the same normal distribution with mean 0 and variance $\mathbb{V}\textrm{ar}[F(\hat{x}^{\epsilon}(T),\xi)]$. It then concludes our desired result.
\end{proof}

\section{The time-stepping method}\label{se:time}

We now adopt the time-stepping method for solving problem \eqref{OCTCCS-sample} with a fixed sample $\{\xi_1,\ldots, \xi_\nu\}$, which uses a finite-difference formula to approximate the time derivative $\dot{x}$. It begins with the division of the time interval $[0,T]$ into $N$ subintervals for a fixed step size $h={T}/{N}=t_{i+1}-t_{i}$ where $i=0,\cdot\cdot\cdot,N-1$.
 Starting from $\mathbf{x}_0^\nu=x_0$, we compute two finite sets of vectors
$
\{\mathbf{x}^{\epsilon,\nu}_{1},\mathbf{x}^{\epsilon,\nu}_{2},\cdot\cdot\cdot,\mathbf{x}^{\epsilon,\nu}_{N}\}\subset \mathbb{R}^n$ and $ \{\mathbf{u}^{\epsilon,\nu}_{1},\mathbf{u}^{\epsilon,\nu}_{2},\cdot\cdot\cdot,\mathbf{u}^{\epsilon,\nu}_{N}\}\subset\mathbb{R}^m
$
in the following manner:
\begin{equation}\label{OCTCCS-sample-time}
\begin{aligned}
& \min_{\{\mathbf{x}_i,\mathbf{u}_i\}_{i=1}^N}\,\,\frac{1}{\nu}\sum_{\ell=1}^\nu F(\mathbf{x}_N,\xi_\ell)+\frac{h}{2}\sum_{i=1}^N\left(\|\mathbf{x}_i-x_{d,i}\|^2+{\delta}\|\mathbf{u}_i-u_{d,i}\|^2\right)&\\
&\textrm{s.t.}
\left\{
\begin{aligned}
&\left.
\begin{aligned}
&\mathbf{x}_{i+1}-\mathbf{x}_i=hA \mathbf{x}_{i+1}+hB\mathbf{u}_{i+1}, &\\
& C \mathbf{x}_{i+1}+D\mathbf{u}_{i+1}-f_{i+1}\leq 0,
\end{aligned}\right\}i=0,1,\cdot\cdot\cdot,N-1,&\\
&\left\|\min \left\{\mathbf{x}_N, \frac{1}{\nu}\sum_{\ell=1}^\nu[M(\xi_\ell)\mathbf{x}_N+q(\xi_\ell)]\right\}\right\|\le \epsilon, &\\
& \frac{1}{\nu}\sum_{\ell=1}^\nu g(\mathbf{x}_N,\xi_\ell)\in K^\epsilon,
\end{aligned}
\right.
\end{aligned}
\end{equation}
where $\epsilon>0$ is a sufficiently small number, $x_{d,i}=x_d(t_i)$, $u_{d,i}=u_d(t_i)$ and $f_i=f(t_{i})$ for $i\in [N]$.

\begin{theorem}\label{existence-discrete}
Suppose that the conditions of Theorem \ref{feasible-set} hold, then for any $\epsilon>0$, problem \eqref{OCTCCS-sample-time} has an optimal solution for sufficiently large $\nu$ and sufficiently small $h$ w.p.1.
\end{theorem}
\begin{proof}
Theorem \ref{feasible-set-1} has shown that the solution set of \eqref{mixedCP-SAA} with any $\epsilon>0$ is nonempty for sufficiently large $\nu$ w.p.1. About the existence of feasible solution to problem \eqref{OCTCCS-sample-time}, it suffices to show that the following problem has a solution for any $b\in\mathbb{R}^n$,
\begin{equation}\label{discrete-iteration-controllability}
\left\{
\begin{aligned}
&\left.
\begin{aligned}
&\mathbf{x}_{i+1}=\mathbf{x}_i+hA \mathbf{x}_{i+1}+hB\mathbf{u}_{i+1}, &\\
& C \mathbf{x}_{i+1}+D\mathbf{u}_{i+1}-f_{i+1}\leq0,
\end{aligned}\right\}i=0,1,\cdot\cdot\cdot,N-1,&\\
&\mathbf{x}_0=x_0,\quad \quad \mathbf{x}_N=b.
\end{aligned}
\right.
\end{equation}

Firstly, denote $A_h=I-h(A-BD^\dag C)$. For sufficiently small $h$,
$A_h$ is nonsingular.
Similar with the proof of Theorem \ref{feasible-set}, from
$\mathbf{x}_{i+1}=\mathbf{x}_{i}+h(A-BD^\dag C)\mathbf{x}_{i+1} +hBD^\dag p_{i+1}$,  the following iteration with $\mathbf{x}_0= x_0$,
\[\mathbf{x}_{i+1}=A_h^{-1}(\mathbf{x}_i+hBD^\dag p_{i+1}), \quad i=0,1,\cdot\cdot\cdot,N-1\]
generates a solution $\{\bar{\mathbf{x}}_i\}_{i=1}^N$
of the system with  $\mathbf{x}_0=x_0,$
$$
\mathbf{x}_{i+1}=\mathbf{x}_i+hA \mathbf{x}_{i+1}+hB\mathbf{u}_{i+1}, \quad C \mathbf{x}_{i+1}+D\mathbf{u}_{i+1}=p_{i+1},\quad
i=0,1,\cdot\cdot\cdot,N-1,
$$
for any given $p_i\in \mathbb{R}^l$, $i=1,\dots,N.$

From Assumption \ref{assumption-Y} and the nonsingularity of $A_h$, we know that the matrix
\[\mathcal{\tilde{R}}_d:=[BY\,\,\,\,A_hBY\,\,\cdot\cdot\cdot\,\,A_h^{n-1}BY]\]
has full row rank $n$. Hence
the matrix
\[\mathcal{R}_d:=[hA_h^{-1}BY\,\,\,\,h(A_h^{-1})^2BY\,\,\cdot\cdot\cdot\,\,h(A_h^{-1})^{n}BY]\]
has full row rank $n$.
According to Theorem 3.1.1 in \cite{book-discrete-controllability}, the  system with $\mathbf{x}_0=0$,
\begin{equation*}
\left\{
\begin{aligned}
&\mathbf{x}_{i+1}=A_h^{-1}(\mathbf{x}_i+hBY {v}_{i+1}),\,\,i=0,1,\cdot\cdot\cdot,N-1,&\\
&\mathbf{x}_N=b-\bar{\mathbf{x}}_N,
\end{aligned}
\right.
\end{equation*}
admits a solution $\{\tilde{\mathbf{x}}_i,\tilde{v}_i\}_{i=1}^N$ for any $b\in\mathbb{R}^n$. Therefore, $\{\tilde{\mathbf{x}}_i+\bar{\mathbf{x}}_i,\tilde{v}_i\}_{i=1}^N$ is a solution of the following equation
\begin{equation*}
\left\{
\begin{aligned}
&\mathbf{x}_{i+1}=A_h^{-1}(\mathbf{x}_i+hBY v_{i+1}+hBD^\dag p_{i+1}),\,\,i=0,1,\cdot\cdot\cdot,N-1,&\\
&\mathbf{x}_0=x_0,\quad \quad \mathbf{x}_N=b.
\end{aligned}
\right.
\end{equation*}
Let $\tilde{\mathbf{u}}_i=Y\tilde{v}_i+D^\dag (p_i-C(\tilde{\mathbf{x}}_i+\bar{\mathbf{x}}_i))$. Then
it is easy to verify that $\{\tilde{\mathbf{x}}_i+\bar{\mathbf{x}}_i,\tilde{\mathbf{u}}_i\}_{i=1}^N$ is a solution of \eqref{discrete-iteration-controllability}
by setting $p_{i}=f_i+\tilde{p}_i$ for any $\tilde{p}_i\leq 0$.

Since $\mathbb{E}[ F(\cdot,\xi)]$ is bounded from below, we can also obtain that $\frac{1}{\nu}\sum_{\ell=1}^\nu F(\cdot,\xi_\ell)$ is also bounded from below with sufficiently large $\nu$. Similar to Theorem 5.1 in \cite{FP2018}, we can prove a minimizing sequence tends to an optimal solution of \eqref{OCTCCS-sample-time}, which shows the existence of optimal solutions to \eqref{OCTCCS-sample-time} with any $\epsilon>0$ for sufficiently large $\nu$ and sufficiently small $h$.
\end{proof}

Let $\{\mathbf{x}^{\epsilon,\nu}_i,\mathbf{u}^{\epsilon,\nu}_i\}_{i=1}^N$ be a solution of \eqref{OCTCCS-sample-time}. We define a piecewise linear function ${x}^{\epsilon,\nu}_h$ and a piecewise constant function ${u}^{\epsilon,\nu}_h$ on $[0,T]$ as below:
\begin{eqnarray}\label{feasible-interpolant}
      {x}^{\epsilon,\nu}_h(t) = \mathbf{x}^{\epsilon,\nu}_{i}+\frac{t-t_{i}}{h}(\mathbf{x}^{\epsilon,\nu}_{i+1}
   -\mathbf{x}^{\epsilon,\nu}_{i}), \quad
  {u}^{\epsilon,\nu}_{h}(t) = \mathbf{u}^{\epsilon,\nu}_{i+1},\,\,\,\,\forall\,\, t\in(t_{i},t_{i+1}].
 \end{eqnarray}
Denote $\hat{\mathcal{S}}^{\epsilon,\nu}_h$ the set of $(\hat{x}_h^{\epsilon,\nu},\hat{u}_h^{\epsilon,\nu})\in H^1(0,T)^n\times L^2(0,T)^m$, where $(\hat{x}^{\epsilon,\nu}_h,\hat{u}^{\epsilon,\nu}_h)$ are defined in \eqref{feasible-interpolant} based on an optimal solution $\{\hat{\mathbf{x}}^{\epsilon,\nu}_i,\hat{\mathbf{u}}^{\epsilon,\nu}_i\}_{i=1}^N$ of \eqref{OCTCCS-sample-time}.  Let
\[\Phi^{\nu}_h(x^{\epsilon,\nu}_h,u^{\epsilon,\nu}_h)=\frac{1}{\nu}\sum_{\ell=1}^\nu F(\mathbf{x}^{\epsilon,\nu}_N,\xi_\ell)+\frac{h}{2}\sum_{i=1}^N
\left(\|\mathbf{x}^{\epsilon,\nu}_i-x_{d,i}\|^2+{\delta}\|\mathbf{u}^{\epsilon,\nu}_i-u_{d,i}\|^2\right).\]
\begin{theorem}
 Suppose that the conditions of Theorem \ref{feasible-set} hold, then we have
\[\lim_{\epsilon\downarrow0}\lim_{\nu\rightarrow\infty}\lim_{h\downarrow0}
\mathbb{D}(\hat{\mathcal{S}}^{\epsilon,\nu}_h,\hat{\mathcal{S}})=0,\,\,\textrm{ w.p.1.}\]
\end{theorem}
\begin{proof}
Following \cite{dontchev,Morfuk2007,polak2002} and the references therein, we can obtain that $\Phi^{\nu}_h$ epiconverges to $\Phi^{\nu}$ as $h\downarrow  0$  over a compact subset of $ H^1(0,T)^n\times L^2(0,T)^m$.

Let $\{\hat{\mathbf{x}}^{\epsilon,\nu}_i,\hat{\mathbf{u}}^{\epsilon,\nu}_i\}_{i=1}^N$ be an optimal solution of \eqref{OCTCCS-sample-time}, which means the boundedness of $\{\hat{u}^{\epsilon,\nu}_{h_k}\}_{k=1}^\infty\subseteq L^2(0,T)^m$. Since $L^2(0,T)^m$ is reflexive, there is a subsequence of $\{\hat{u}^{\epsilon,\nu}_{h_k}\}$, which we may assume without loss of generality to be $\{\hat{u}^{\epsilon,\nu}_{h_k}\}$ itself,  having a weak limit $\hat{u}^{\epsilon,\nu}_*\in L^2(0,T)^m$.
It is easy to see that $(\hat{x}^{\epsilon,\nu}_{h_k},\hat{u}^{\epsilon,\nu}_{h_k})$ satisfies the differential equation $\dot{x}^{\epsilon,\nu}_{h_k}(t)=A\mathbf{x}^{\epsilon,\nu}_{i+1}+Bu^{\epsilon,\nu}_{h_k}(t)$ for a.e. $t\in(t_i,t_{i+1})$ with some $i\in [N]$.
Therefore, there is $\hat{x}^{\epsilon,\nu}_*\in H^1(0,T)^n$ such that $\hat{x}^{\epsilon,\nu}_h\rightarrow \hat{x}^{\epsilon,\nu}_*$ in $H^1(0,T)^n$ by $\hat{u}^{\epsilon,\nu}_h\rightarrow\hat{u}^{\epsilon,\nu}_*$ in $ L^2(0,T)^m$. According to Theorem 2.5 in \cite{Wets1981}, we can obtain $\lim_{h\downarrow0}\mathbb{D}(\hat{\mathcal{S}}^{\epsilon,\nu}_h,\hat{\mathcal{S}}^{\epsilon,\nu})=0$ with some $\epsilon>0$ and sufficiently large $\nu$ and then $\lim_{\epsilon\downarrow0}\lim_{\nu\rightarrow\infty}\lim_{h\downarrow0}
\mathbb{D}(\hat{\mathcal{S}}^{\epsilon,\nu}_h,\hat{\mathcal{S}})=0$ w.p.1.
\end{proof}

\subsection{Error estimates of optimal values of problem \eqref{OCTCCS-sample-time} to problem \eqref{OCTCCS-sample}}

In this subsection,
 we investigate the Euler approximation of problem \eqref{OCTCCS-sample}.
 Our results are related to the Euler approximation of the optimal control problem with two-point differential system \cite[Theorem 5]{Dontchev-1994}, which requires the convexity of the terminal set. However, the terminal constraint set $\mathcal{Z}^{\epsilon,\nu}$ in \eqref{OCTCCS-sample} is generally nonconvex due to the existence of the complementarity constraints.

\begin{lemma}\label{error-discrete-continuous}
Suppose that the conditions of Theorem \ref{feasible-set} hold.
Let $({x}^{\epsilon,\nu}_h,{u}^{\epsilon,\nu}_h)$ be defined in \eqref{feasible-interpolant} by a feasible solution $\{\mathbf{x}^{\epsilon,\nu}_i,\mathbf{u}^{\epsilon,\nu}_i\}_{i=1}^N$ of  \eqref{OCTCCS-sample-time}. Then, for sufficiently small $h$, there is a feasible solution $(x^{\epsilon,\nu},u^{\epsilon,\nu})$ of problem \eqref{OCTCCS-sample} such that
\[\|x^{\epsilon,\nu}-{x}^{\epsilon,\nu}_h\|_{L^2}\leq \sqrt{T}C_y h,\,\,\,\|x^{\epsilon,\nu}-{x}^{\epsilon,\nu}_h\|_{H^1}\leq \sqrt{1+T}C_y h,\,\,\,\|u^{\epsilon,\nu}-{u}^{\epsilon,\nu}_h\|_{L^2}\leq C_uh,\]
where $C_y$ and $C_u$ are two constants independent of $h$.
\end{lemma}
\begin{proof}
We denote two positive constants $\theta_x$ and $\theta_u$ such that $\max_{i\in[N]}\|\mathbf{x}^{\epsilon,\nu}_i\|\leq \theta_x$ and $\max_{i\in[N]}\|\mathbf{u}^{\epsilon,\nu}_i\|\leq \theta_u$.
According to Theorem \ref{existence-discrete},   there are $v_i\in \mathbb{R}^{m-l}$ and $\tilde{p}_i\leq 0$ such that $\mathbf{u}^{\epsilon,\nu}_i=Y{v}_i+D^\dag (f_i+\tilde{p}_i-C\mathbf{x}^{\epsilon,\nu}_i)$ for $i\in [N].$ Let $x^{\epsilon,\nu}(t)$ be the solution of the following system, for $t\in(t_i,t_{i+1}]$,
\begin{equation*}
\left\{
\begin{aligned}
& \dot{x}^{\epsilon,\nu}(t)=(A-BD^\dag C)x^{\epsilon,\nu}(t)+BY({v}_{i+1}+{a_{i+1}(t-t_i)})+BD^\dag (\tilde{p}_{i+1}+f(t)),&\\
&  {x}^{\epsilon,\nu}(0)=x_0,\,\,\,x^{\epsilon,\nu}(T)=\mathbf{x}^{\epsilon,\nu}_N,
\end{aligned}
\right.
\end{equation*}
where $\{a_i\}_{i=1}^{N}\subset \mathbb{R}^{m-l}$ fulfills
\begin{eqnarray*}
& &\mathbf{x}^{\epsilon,\nu}_N=e^{(A-BD^\dag C)T}x_0+\sum_{i=0}^{N-1}\left[\int_{t_i}^{t_{i+1}}e^{(A-BD^\dag C)(T-\tau)}d\tau B(Yv_{i+1}+D^\dag \tilde{p}_{i+1})\right.\\
& &\left.+\int_{t_i}^{t_{i+1}}e^{(A-BD^\dag C)(T-\tau)}BD^\dag f(\tau)d\tau +\int_{t_i}^{t_{i+1}}e^{(A-BD^\dag C)(T-\tau)}BYa_{i+1}(\tau-t_i)d\tau\right].
\end{eqnarray*}
In addition, we know that $x_h^{\epsilon,\nu}$ solves the differential equation, for any $t\in(t_i,t_{i+1}]$,
\begin{eqnarray*}
\left\{
\begin{aligned}
&\dot{x}^{\epsilon,\nu}_h(t)=(A-BD^\dag C)x^{\epsilon,\nu}_h(t)+BY({v}_{i+1}+a_{i+1}(t-t_i))+BD^\dag (\tilde{p}_{i+1}+f(t))+y(t),&\\
& {x}^{\epsilon,\nu}_h(0)=x_0,\,\,\,x^{\epsilon,\nu}_h(T)=\mathbf{x}^{\epsilon,\nu}_N,
\end{aligned}
\right.
\end{eqnarray*}
where $y(t)=(A-BD^\dag C)(\mathbf{x}^{\epsilon,\nu}_{i+1}-x_h^{\epsilon,\nu}(t))-BYa_{i+1}(t-t_i)-BD^\dag(f(t)-f_{i+1})$.
Since $f\in L^2(0,T)^l$, there is a $h_0>0$ such that $\|f(t)-f_{i+1}\|\leq h$ for any $h\in(0,h_0]$ and a.e. $t\in(t_i,t_{i+1}]$.
Let $\tilde{f}(t)=f_{i+1}$ for $t\in(t_i,t_{i+1}]$, we then have $\|f-\tilde{f}\|_{L^2}\leq\sqrt{T}h$.
It means that $\|y\|_{L^2} \leq  C_yh$ for any $h\in(0,h_0]$, where $$C_y=\|A-BD^\dag C\|(\|A\|\theta_x+\|B\|\theta_u)\sqrt{T}+\max_{i\in[N]}\|a_i\|\|BY\|\sqrt{\frac{T}{3}}+\|BD^\dag\|\sqrt{T}.$$
Therefore, we have, for any $t\in(t_i,t_{i+1}]$,
\[\|x^{\epsilon,\nu}-{x}^{\epsilon,\nu}_h\|^2_{L^2}\leq \int_0^T\int_0^t\|\dot{x}^{\epsilon,\nu}(\tau)-\dot{x}^{\epsilon,\nu}_h(\tau)\|^2d\tau dt= \int_0^T\int_0^t\|y(\tau)\|^2d\tau dt\leq \|y\|_{L^2}^2 T.\]
Hence, according to the definition of $\|\cdot\|_{H^1}$, we obtain that
\[\|x^{\epsilon,\nu}-{x}^{\epsilon,\nu}_h\|_{H^1}\leq \sqrt{1+T}\|y\|_{L^2}\leq \sqrt{1+T}C_y h.\]

Let $u^{\epsilon,\nu}(t)=Y(v_{i+1}+a_{i+1}(t-t_i))+D^\dag(\tilde{p}_{i+1}+f(t)-Cx^{\epsilon,\nu}(t))$ for any $t\in(t_i,t_{i+1}]$. It is clear that $u^{\epsilon,\nu}_h(t)=\mathbf{u}^{\epsilon,\nu}_{i+1}=Yv_{i+1}+D^\dag(f_{i+1}
+\tilde{p}_{i+1}-C\mathbf{x}^{\epsilon,\nu}_{i+1})$ for any $t\in(t_i,t_{i+1}]$. Then we have $\|u^{\epsilon,\nu}-u^{\epsilon,\nu}_h\|_{L^2}\leq C_u h,$
where $$C_u=\max_{i\in[N]}\|a_i\|\|Y\|\sqrt{\frac{T}{3}}+\sqrt{T}\|D^\dag\|+\|D^\dag C\|(C_y+\|A\|\theta_x+\|B\|\theta_u)\sqrt{T}.$$

Clearly, according to the definition of $u^\nu(t)$, we can obtain that for any $t\in(t_i,t_{i+1}]$, \[Cx^{\epsilon,\nu}(t)+Du^{\epsilon,\nu}(t)-f(t)=\tilde{p}_{i+1}\leq 0,\]
which shows that $(x^{\epsilon,\nu},u^{\epsilon,\nu})$ is a feasible solution of problem \eqref{OCTCCS-sample}.
\end{proof}

\begin{lemma}\label{error-continuous-discrete}
Suppose that the conditions of Theorem \ref{feasible-set} hold.
Let $(x^{\epsilon,\nu},u^{\epsilon,\nu})$ be a feasible solution of problem \eqref{OCTCCS-sample} with $\|x^{\epsilon,\nu}(t)\|\leq \theta'_x$ and $\|u^{\epsilon,\nu}(t)\|\leq \theta'_u$ for a.e. $t\in[0,T]$, where $\theta'_x$ and $\theta'_u$ are two positive constants. Then, for sufficiently small $h$, there is $(x^{\epsilon,\nu}_h,{u}^{\epsilon,\nu}_h)$ defined in \eqref{feasible-interpolant} by a feasible solution $\{\mathbf{x}^{\epsilon,\nu}_i,\mathbf{u}^{\epsilon,\nu}_i\}_{i=1}^N$ of \eqref{OCTCCS-sample-time}, such that
\[\|x^{\epsilon,\nu}-{x}^{\epsilon,\nu}_h\|_{L^2}\leq \sqrt{T}\tilde{C}_y h,\,\,\,\|x^{\epsilon,\nu}-{x}^{\epsilon,\nu}_h\|_{H^1}\leq \sqrt{1+T}\tilde{C}_y h,\,\,\,\|u^{\epsilon,\nu}-{u}^{\epsilon,\nu}_h\|_{L^2}\leq \tilde{C}_uh,\]
where $\tilde{C}_y$ and $\tilde{C}_u$ are two positive constants independent of $h$.
\end{lemma}
\begin{proof}
Let $(x^{\epsilon,\nu},u^{\epsilon,\nu})\in H^1(0,T)^n\times L^2(0,T)^m$ be a feasible solution of problem \eqref{OCTCCS-sample}, then there are $v\in L^2(0,T)^{m-l}$ and $\tilde{p}\in L^2(0,T)^l$ with $\tilde{p}(t)\leq 0$ for a.e. $t\in[0,T]$ such that $u^{\epsilon,\nu}(t)=Yv(t)+D^\dag (\tilde{p}(t)+f(t)-Cx^{\epsilon,\nu}(t))$. In addition, there are $h_1>0$ and a piecewise constant function $\varphi_v(t)=\frac{1}{h}\int_{t_i}^{t_{i+1}}v(\tau)d\tau:={\varphi}_{i+1}$ for any $t\in(t_i,t_{i+1}]$ such that $\|v(t)-\varphi_v(t)\|\leq h$ for a.e. $t\in(t_i,t_{i+1}]$ with $h\in(0,h_1]$. There are also $h_2>0$ and a piecewise constant function $\varphi_p(t)=\frac{1}{h}\int_{t_i}^{t_{i+1}}\tilde{p}(\tau)d\tau:=\tilde{\varphi}_{i+1}$ for any $t\in(t_i,t_{i+1}]$ such that $\|\tilde{p}(t)-\varphi_p(t)\|\leq h$ with $h\in(0,h_2]$ and $\varphi_p(t)\leq0$ for a.e. $t\in(t_i,t_{i+1}]$.

Recall $A_h= I-h(A-BD^\dag C)$.  For $i=0,1,\cdot\cdot\cdot,N-1$, let $\mathbf{x}^{\epsilon,\nu}_{0}=x_0$ and
\begin{equation*}
\begin{aligned}
\mathbf{x}^{\epsilon,\nu}_{i+1}=A_h^{-1}(\mathbf{x}^{\epsilon,\nu}_{i}+hBY(\varphi_{i+1}
+a_{i+1}h)+hBD^\dag(\tilde{\varphi}_{i+1}+f_{i+1})),
\end{aligned}
\end{equation*}
 where  $\{a_i\}_{i=1}^{N}\subset \mathbb{R}^{m-l}$  fulfills
\[x^{\epsilon,\nu}(T)=A_h^{-N}x_0+h\sum_{i=0}^{N-1} A_h^{-(i+1)}[BY(\varphi_{i+1}+a_{i+1}h)+hBD^\dag(\tilde{\varphi}_{i+1}+f_{i+1})].\]

Let $\mathbf{u}^{\epsilon,\nu}_{i}=Y(\varphi_{i}+a_{i}h)+D^\dag(\tilde{\varphi}_{i}+f_{i}-C \mathbf{x}^{\epsilon,\nu}_{i})$ for $i\in [N]$.
Since $\|x^{\epsilon,\nu}(t)\|\leq \theta'_x$ and $\|u^{\epsilon,\nu}(t)\|\leq \theta'_u$ for a.e. $t\in[0,T]$, there is a partition to $[0,T]$ such that the sequences
$\{Y\varphi_{i},D^\dag\tilde{\varphi}_{i}\}_{i=1}^{N}$ and $\{\mathbf{x}^{\epsilon,\nu}_i,\mathbf{u}^{\epsilon,\nu}_i\}_{i=1}^N$ are also bounded for any given $N$. We denote that $\tilde{\theta}_x$ and $\tilde{\theta}_u$ are two positive constants such that $\max_{i\in[N]}\|\mathbf{x}^{\epsilon,\nu}_i\|\leq \tilde{\theta}_x$ and $\max_{i\in[N]}\|\mathbf{u}^{\epsilon,\nu}_i\|\leq \tilde{\theta}_u$.
 It is clear that $x^{\epsilon,\nu}_h$ satisfies,
\[\dot{x}^{\epsilon,\nu}_h(t)=(A-BD^\dag C)x^{\epsilon,\nu}_h(t)+BY v(t)+BD^\dag (\tilde{p}(t)+f(t))+\tilde{y}(t),\,\,\,t\in(t_i,t_{i+1}],\]
where $\tilde{y}(t)=(A-BD^\dag C)(\mathbf{x}^{\epsilon,\nu}_{i+1}-x^{\epsilon,\nu}_h(t))+BY(\varphi_{i+1}+a_{i+1}h-v(t))+BD^\dag (\tilde{\varphi}_{i+1}+f_{i+1}-\tilde{p}(t)-f(t))$. It means that $\|\tilde{y}\|_{L^2}\leq \tilde{C}_yh $ for any $h\in(0,\min\{h_0,h_1,h_2\}]$, where
$$\tilde{C}_y=\|A-BD^\dag C\|(\|A\|\tilde{\theta}_x+\|B\|\tilde{\theta}_u)\sqrt{T}+\|BY\|(\sqrt{T}
+\max_{i\in[N]}\|a_i\|)+2\|BD^\dag \|\sqrt{T}.$$
Hence  $\|x^{\epsilon,\nu}-{x}^{\epsilon,\nu}_h\|_{L^2}\leq\sqrt{T}\|\tilde{y}\|_{L^2}\leq \sqrt{T}\tilde{C}_y h$ and $\|x^{\epsilon,\nu}-{x}^{\epsilon,\nu}_h\|_{H^1}\leq \sqrt{1+T}\tilde{C}_y h.$
Moreover, we have
$\|u^{\epsilon,\nu}-u^{\epsilon,\nu}_h\|_{L^2}\leq \tilde{C}_uh,$
where
$$\tilde{C}_u=\|Y\|(\sqrt{T}+\max_{i\in [N]}\|a_i\|)+\|C\|\sqrt{T}(\|A\|\tilde{\theta}_x+\|B\|\tilde{\theta}_u+\tilde{C}_y)+2\|D^\dag\| \sqrt{T}.$$

Obviously, from the definition of $\mathbf{u}^{\epsilon,\nu}_{i}$, we get $C \mathbf{x}^{\epsilon,\nu}_{i}+D\mathbf{u}^{\epsilon,\nu}_{i}-f_{i}=\tilde{\varphi}_{i}\leq0$  $(i\in [N])$, which means that $\{\mathbf{x}^{\epsilon,\nu}_i,\mathbf{u}^{\epsilon,\nu}_i\}_{i=1}^N$ is a feasible solution of \eqref{OCTCCS-sample-time}.
\end{proof}

\begin{lemma}\label{error-values}
Suppose that the conditions of Theorem \ref{feasible-set} hold.
Let $\{\mathbf{x}^{\epsilon,\nu}_i,\mathbf{u}^{\epsilon,\nu}_i\}_{i=1}^N$ be a feasible solution of \eqref{OCTCCS-sample-time} with
$\max_{i\in[N]}\|\mathbf{x}^{\epsilon,\nu}_i\|\leq \bar{\theta}_x$ and $\max_{i\in[N]}\|\mathbf{u}^{\epsilon,\nu}_i\|\leq \bar{\theta}_u$, where $\bar{\theta}_x$ and $\bar{\theta}_u$ are two positive constants, and let  $({x}^{\epsilon,\nu}_h,{u}^{\epsilon,\nu}_h)$ be defined in \eqref{feasible-interpolant}. Then, for sufficiently small $h$,
\[|\Phi^{\nu}(x^{\epsilon,\nu}_h,u^{\epsilon,\nu}_h)
-\Phi^{\nu}_h(x^{\epsilon,\nu}_h,u^{\epsilon,\nu}_h)|\leq C_Th,\]
where $C_T$ is a positive constant  independent of $h$.
\end{lemma}
\begin{proof}
Since $\{\mathbf{x}^{\epsilon,\nu}_i,\mathbf{u}^{\epsilon,\nu}_i\}_{i=1}^N$ is a bounded feasible solution of \eqref{OCTCCS-sample-time}, $\Phi^{\nu}_h(x^{\epsilon,\nu}_h,u^{\epsilon,\nu}_h)$ is bounded, which means that there is a $\theta_o>0$ such that $\max_{i\in[N]}\{\|\mathbf{x}^{\epsilon,\nu}_i-x_{d,i}\|, \|\mathbf{u}_i^{\epsilon,\nu}-u_{d,i}\|\}\leq \theta_o$. Therefore, we have $\Phi^\nu(x^{\epsilon,\nu}_h,u^{\epsilon,\nu}_h)-\Phi^{\nu}_h(x^{\epsilon,\nu}_h,u^{\epsilon,\nu}_h)= W_1+W_2$,
where
\begin{equation*}
\begin{aligned}
W_1&=\frac{1}{2}\sum_{i=0}^{N-1} \int_{t_i}^{t_{i+1}}\left(\|x^{\epsilon,\nu}_h(t)-x_d(t)\|^2-\|\mathbf{x}^{\epsilon,\nu}_{i+1}-x_d(t_{i+1})\|^2 \right)dt,&\\
W_2&=\frac{\delta}{2}\sum_{i=0}^{N-1} \int_{t_i}^{t_{i+1}}\left(\|u^{\epsilon,\nu}_h(t)-u_d(t)\|^2-\|\mathbf{u}^{\epsilon,\nu}_{i+1}-u_d(t_{i+1})\|^2 \right)dt.
\end{aligned}
\end{equation*}
Note that $x_d\in L^2(0,T)^n$ implies that there is $h_x>0$ such that $\|x_d(t_{i+1})-x_d(t)\|\leq h$ for a.e. $t\in(t_i,t_{i+1}]$ with $h\in(0,h_x]$. Then we have
\begin{eqnarray*}
& &|W_1|\leq \frac{1}{2}\sum_{i=0}^{N-1} \int_{t_i}^{t_{i+1}}(\|x^{\epsilon,\nu}_h(t)-\mathbf{x}^{\epsilon,\nu}_{i+1}\|\\
& &+\|x_d(t_{i+1})-x_d(t)\|)\left(\|x^{\epsilon,\nu}_h(t)-\mathbf{x}^{\epsilon,\nu}_{i+1}\|
+\|x_d(t_{i+1})-x_d(t)\|+2\|\mathbf{x}^{\epsilon,\nu}_{i+1}-x_d(t_{i+1})\| \right)dt\\
& &\leq \frac{1}{2}\sum_{i=0}^{N-1} (\|A\|\bar{\theta}_x+\|B\|\bar{\theta}_u+1)((\|A\|\bar{\theta}_x+\|B\|\bar{\theta}_u+1)h+2\theta_o)h^2\\
& &\leq  \frac{1}{2}(\|A\|\bar{\theta}_x+\|B\|\bar{\theta}_u+1)((\|A\|\bar{\theta}_x+\|B\|\bar{\theta}_u+1)h_x+2\theta_o)Th.
\end{eqnarray*}
Moreover, $u_d\in L^2(0,T)^m$ implies that there is $h_u>0$ such that $\|u_d(t_{i+1})-u_d(t)\|\leq h$ for a.e. $t\in(t_i,t_{i+1}]$ with $h\in(0,h_u]$. Then
\begin{eqnarray*}
|W_2|&\leq& \frac{\delta}{2}\sum_{i=0}^{N-1} \int_{t_i}^{t_{i+1}}\|u_d(t_{i+1})-u_d(t)\| \left(\|u_d(t_{i+1})-u_d(t)\|+2\|\mathbf{u}^{\epsilon,\nu}_{i+1}-u_d(t_{i+1})\| \right)dt\\
&\leq & \frac{\delta}{2} (h_u+2\theta_o)Th.
\end{eqnarray*}
It then derives our result for $h\in(0,\min\{h_x,h_u\}]$ by denoting \[C_T=\frac{1}{2}(\|A\|\bar{\theta}_x+\|B\|\bar{\theta}_u+1)
((\|A\|\bar{\theta}_x+\|B\|\bar{\theta}_u+1)h_x+2\theta_o)T+\frac{\delta}{2} (h_u+2\theta_o)T.\]
\end{proof}

\begin{theorem}\label{error-optimal-values}
Suppose that the conditions of Theorem \ref{feasible-set} hold. Let $(\hat{x}^{\epsilon,\nu},\hat{u}^{\epsilon,\nu})$ be an optimal solution of \eqref{OCTCCS-sample}, and let $(\hat{x}^{\epsilon,\nu}_h,\hat{u}^{\epsilon,\nu}_h)$ be defined in \eqref{feasible-interpolant} associated with an optimal solution
$\{\hat{\mathbf{x}}^{\epsilon,\nu}_i,\hat{\mathbf{u}}^{\epsilon,\nu}_i\}_{i=1}^N$ of \eqref{OCTCCS-sample-time}. Then, for sufficiently small $h$,
\begin{equation}\label{error1}
|\Phi^{\nu}_h(\hat{x}^{\epsilon,\nu}_h,\hat{u}^{\epsilon,\nu}_h)
-\Phi^{\nu}(\hat{x}^{\epsilon,\nu},\hat{u}^{\epsilon,\nu})|\leq \hat{C}_Th,
\end{equation}
where $C_T$ is a positive constant  independent of $h$.
\end{theorem}
\begin{proof}
Since $\{\hat{\mathbf{x}}^{\epsilon,\nu}_i,\hat{\mathbf{u}}^{\epsilon,\nu}_i\}_{i=1}^N$ is an optimal solution of \eqref{OCTCCS-sample-time}, there is $\psi_0$ such that $\max\{\|\hat{x}^{\epsilon,\nu}_h-x_d\|_{L^2},\|\hat{u}^{\epsilon,\nu}_h-u_d\|_{L^2}\}\leq \psi_0$, where $(\hat{x}_h^{\epsilon,\nu},\hat{u}_h^{\epsilon,\nu})$ is defined in \eqref{feasible-interpolant} associated with the sequence $\{\hat{\mathbf{x}}^{\epsilon,\nu}_i,\hat{\mathbf{u}}^{\epsilon,\nu}_i\}_{i=1}^N$. Similarly, $(\hat{x}^{\epsilon,\nu},\hat{u}^{\epsilon,\nu})$ is an optimal solution of \eqref{OCTCCS}, which means that there is $\psi_1$ such that $\max\{\|\hat{x}^{\epsilon,\nu}-x_d\|_{L^2},\|\hat{u}^{\epsilon,\nu}-u_d\|_{L^2}\}\leq \psi_1$.

Following Lemma \ref{error-discrete-continuous}-Lemma \ref{error-values}, there is $\bar{h}>0$ such that for any $h\in(0,\bar{h}]$  there is $(x^{\epsilon,\nu},u^{\epsilon,\nu})$, which is a feasible solution of \eqref{OCTCCS-sample} satisfying $\|x^{\epsilon,\nu}-\hat{x}^{\epsilon,\nu}_h\|_{L^2}\leq \sqrt{T} C_y h$ and $\|u^{\epsilon,\nu}-\hat{u}^{\epsilon,\nu}_h\|_{L^2}\leq C_uh.$
Moreover, according to Lemma \ref{error-continuous-discrete}, for any $h\in(0,\bar{h}]$ there is a $\{\mathbf{x}^{\epsilon,\nu}_i,\mathbf{u}^{\epsilon,\nu}_i\}_{i=1}^N$, which is a feasible solution of \eqref{OCTCCS-sample-time}, such that
$\|\hat{x}^{\epsilon,\nu}-{x}^{\epsilon,\nu}_h\|_{L^2}\leq \sqrt{T} \tilde{C}_y h$ and $\|\hat{u}^{\epsilon,\nu}-{u}^{\epsilon,\nu}_h\|_{L^2}\leq \tilde{C}_uh,$
where $(x^{\epsilon,\nu}_h,u^{\epsilon,\nu}_h)$ is defined in \eqref{feasible-interpolant} based on the sequence $\{\mathbf{x}^{\epsilon,\nu}_i,\mathbf{u}^{\epsilon,\nu}_i\}_{i=1}^N$.

Then we have $\Phi^{\nu}_h(\hat{x}^{\epsilon,\nu}_h,\hat{u}^{\epsilon,\nu}_h)\leq \Phi^\nu_h({x}^{\epsilon,\nu}_h,{u}^{\epsilon,\nu}_h)$, which means
\begin{eqnarray*}
& &\Phi^{\nu}_h(\hat{x}^{\epsilon,\nu}_h,\hat{u}^{\epsilon,\nu}_h)
-\Phi^{\nu}(\hat{x}^{\epsilon,\nu},\hat{u}^{\epsilon,\nu})\leq \Phi^{\nu}_h({x}^{\epsilon,\nu}_h,{u}^{\epsilon,\nu}_h)
-\Phi^{\nu}(\hat{x}^{\epsilon,\nu},\hat{u}^{\epsilon,\nu})\\
& &\leq |\Phi^{\nu}_h({x}^{\epsilon,\nu}_h,{u}^{\epsilon,\nu}_h)
-\Phi^{\nu}({x}^{\epsilon,\nu}_h,{u}^{\epsilon,\nu}_h)|
+|\Phi^{\nu}({x}^{\epsilon,\nu}_h,{u}^{\epsilon,\nu}_h)
-\Phi^{\nu}(\hat{x}^{\epsilon,\nu},\hat{u}^{\epsilon,\nu})|.
\end{eqnarray*}
Clearly,
\begin{eqnarray*}
|\Phi^{\nu}({x}^{\epsilon,\nu}_h,{u}^{\epsilon,\nu}_h)
-\Phi^{\nu}(\hat{x}^{\epsilon,\nu},\hat{u}^{\epsilon,\nu})|&\leq & \frac{1}{2}\|{x}^{\epsilon,\nu}_h-\hat{x}^{\epsilon,\nu}\|_{L^2}(\|{x}^{\epsilon,\nu}_h-\hat{x}^{\epsilon,\nu}\|_{L^2}
+2\|\hat{x}^{\epsilon,\nu}-x_d\|_{L^2})\\
& &+\frac{\delta}{2}\|{u}^{\epsilon,\nu}_h-\hat{u}^{\epsilon,\nu}\|_{L^2}
(\|{u}^{\epsilon,\nu}_h-\hat{u}^{\epsilon,\nu}\|_{L^2}+2\|\hat{u}^{\epsilon,\nu}-u_d\|_{L^2})\\
&\leq & \left(\frac{1}{2} \sqrt{T} \tilde{C}_y(\sqrt{T} \tilde{C}_y \bar{h}+2\psi_1)+\frac{\delta}{2}\tilde{C}_u(\tilde{C}_u\bar{h}+2\psi_1) \right) h.
\end{eqnarray*}
Hence, according to Lemma \ref{error-values}, we get
\begin{eqnarray*}
\Phi^{\nu}_h(\hat{x}^{\epsilon,\nu}_h,\hat{u}^{\epsilon,\nu}_h)
 -\Phi^{\nu}(\hat{x}^{\epsilon,\nu},\hat{u}^{\epsilon,\nu})&\leq& \left(\frac{1}{2} \sqrt{T} \tilde{C}_y(\sqrt{T} \tilde{C}_y \bar{h}+2\psi_1)+\frac{\delta}{2}\tilde{C}_u(\tilde{C}_u\bar{h}+2\psi_1)+C_T \right) h\\
 & =:&\beta_1h.
 \end{eqnarray*}

From $\Phi^{\nu}(\hat{x}^{\epsilon,\nu},\hat{u}^{\epsilon,\nu})\leq \Phi^{\nu}({x}^{\epsilon,\nu},{u}^{\epsilon,\nu})$, we have
\begin{eqnarray*}
& &\Phi^{\nu}(\hat{x}^{\epsilon,\nu},\hat{u}^{\epsilon,\nu})
-\Phi_h^{\nu}(\hat{x}^{\epsilon,\nu}_h,\hat{u}^{\epsilon,\nu}_h)\leq \Phi^{\nu}({x}^{\epsilon,\nu},{u}^{\epsilon,\nu})
-\Phi_h^{\nu}(\hat{x}^{\epsilon,\nu}_h,\hat{u}^{\epsilon,\nu}_h)\\
& &\leq |\Phi^{\nu}({x}^{\epsilon,\nu},{u}^{\epsilon,\nu})
-\Phi^{\nu}(\hat{x}^{\epsilon,\nu}_h,\hat{u}^{\epsilon,\nu}_h)|+ |\Phi^{\nu}(\hat{x}^{\epsilon,\nu}_h,\hat{u}^{\epsilon,\nu}_h)
-\Phi_h^{\nu}(\hat{x}^{\epsilon,\nu}_h,\hat{u}^{\epsilon,\nu}_h)|
\end{eqnarray*}
and
\begin{eqnarray*}
|\Phi^{\nu}({x}^{\epsilon,\nu},{u}^{\epsilon,\nu})
-\Phi^{\nu}(\hat{x}^{\epsilon,\nu}_h,\hat{u}^{\epsilon,\nu}_h)|&\leq & \frac{1}{2}\|x^{\epsilon,\nu}-\hat{x}^{\epsilon,\nu}_h\|_{L^2}(\|x^{\epsilon,\nu}-\hat{x}^{\epsilon,\nu}_h\|_{L^2}
+2\|\hat{x}^{\epsilon,\nu}_h-x_d\|_{L^2})\\
& &+\frac{\delta}{2}\|u^{\epsilon,\nu}-\hat{u}^{\epsilon,\nu}_h\|_{L^2}(\|u^{\epsilon,\nu}-\hat{u}^{\epsilon,\nu}_h\|_{L^2}
+2\|\hat{u}^{\epsilon,\nu}_h-u_d\|_{L^2})\\
&\leq & \left(\frac{1}{2} \sqrt{T} {C}_y(\sqrt{T} {C}_y \bar{h}+2\psi_0)+\frac{\delta}{2}{C}_u({C}_u\bar{h}+2\psi_0) \right) h.
\end{eqnarray*}
It means that
\begin{eqnarray*}
\Phi^{\nu}(\hat{x}^{\epsilon,\nu},\hat{u}^{\epsilon,\nu})
-\Phi^{\nu}_h(\hat{x}^{\epsilon,\nu}_h,\hat{u}^{\epsilon,\nu}_h)&\leq& \left(\frac{1}{2} \sqrt{T} {C}_y(\sqrt{T} {C}_y \bar{h}+2\psi_0)+\frac{\delta}{2}{C}_u({C}_u\bar{h}+2\psi_0)+C_T \right) h\\
&=:&\beta_2h.
\end{eqnarray*}
Hence (\ref{error1}) holds with
$\hat{C}_T= \max\{\beta_1,\beta_2\}.$
\end{proof}

\section{Numerical experiments}\label{Example}
We use the following numerical example to illustrate the theoretical results obtained in this paper.
  \begin{equation}\label{numerical-example}
\begin{aligned}
&\min_{x,u}\,\,(\mathbb{E}[\xi_1^2+\xi_2]+ 1)\|x(T)\|^2 +\frac{1}{2}\left(\|x\|_{L^2}^2
+\|u\|_{L^2}^2\right)&\\
&\textrm{s.t.}
\left\{
\begin{aligned}
&\left.
\begin{aligned}
&\dot{x}_1(t)=u_1(t), &\\
&\dot{x}_2(t)=x_2(t)-u_2(t),&\\
&\dot{x}_3(t)=u_3(t), &\\
&\dot{x}_4(t)=x_4(t)-u_4(t),&\\
&x_1(t)+u_2(t)\leq 0,&\\
&{x}_4(t)+u_3(t)\leq 0,&\\
\end{aligned}\right\}a.e.\,\,t\in(0,T),&\\
&x(0)=(1,1,1,1)^\top,\quad  0\leq          x(T)
         \perp \mathbb{E}[M(\xi)
                 x(T)
      + q(\xi)]
      \geq 0,&\\
&\left(x_1(T)+x_3(T), (\mathbb{E}[\xi_1]+1)(x_2(T)+x_4(T)) \right)^\top\in \mathcal{B}(0,\sqrt{6})\subset \mathbb{R}^2,
\end{aligned}
\right.
\end{aligned}
\end{equation}
where
  \[q(\xi)
       =\left(\begin{array}{c}
         3 +\xi_2 \\
         \xi_1\\
         1-\xi_2\\
         \xi_1+1\\
       \end{array}\right)\quad {\rm and} \quad
       M(\xi)=\left( \begin{array}{cccc}
  -2-\xi_1 & 0 & -\xi_2 &-\xi_1 \\
  0 & \xi_2 & -1 &0 \\
  0 & -\xi_1 & \xi_2 & 0 \\
 \xi_2 -1 & 0 & 0 &\xi_1 \\
  \end{array}\right). \]
We set $T=1$, and  $\xi_1\sim {\cal{N}}(1,0.01)$ and $\xi_2\sim
{\cal{U}}(-1,1)$.  It is easy to verify that $\mathbb{E}[M(\xi)]$ is a Z-matrix and the controllability matrix in Assumption \ref{assumption-Y}
\[\mathcal{R}=\left( \begin {array}{cccccccc} 0&1&0&0&0&0&0&0\\
0&0&0&1&0&1&0&1\\
0&0&1&0&1&0&1&0
\\
-1&0&-1&0&-1&0&-1&0\end {array} \right),
\]
is full row rank.
We can derive that the solution set of the LCP in \eqref{numerical-example} is
$$\{(0,0,0,0)^\top,(1,0,0,0)^\top,(0,1,1,0)^\top,(1,1,1,0)^\top\}$$
and the solution set of the terminal constraints in \eqref{numerical-example} is $$\{(0,0,0,0)^\top,(1,0,0,0)^\top,(0,1,1,0)^\top\}.$$

By computation using Maple for these three terminal vectors, we obtain that the values of the objective function in \eqref{numerical-example} with $(0,0,0,0)^\top$, $(1,0,0,0)^\top$ and $(0,1,1,0)^\top$ are $35.29712213$, $48.71659453$ and $25.17501124$, respectively, which means that the optimal solution of \eqref{numerical-example} has a unique terminal vector $x(T)=(0,1,1,0)^\top$.
With the terminal vector   $(0,1,1,0)^\top$, we derive  an optimal solution
 \begin{equation*}
 \begin{aligned}
 x_1^*(t)=&(-40.3067\sin(at)+0.3685\cos(at))e^{-ct}
+(1.3063\sin(at)+0.6315\cos(at))e^{ct}\\
x_2^*(t)=&(17.379\sin(at)+2.4445\cos(at))e^{-ct}
+(3.0042\sin(at)-1.4445\cos(at))e^{ct}\\
x_3^*(t)=&2.0488e^{-1.618t}+1.8734e^{1.618t}-0.2901e^{0.61805t}-2.6321e^{-0.61805t}\\
x_4^*(t)=&3.315e^{-1.618t}+0.46938e^{1.618t}-1.1578e^{0.61805t}-1.6266e^{-0.61805t}\\
u_1^*(t)=&(51.113\sin(at)-14.198\cos(at))e^{-ct}
+(1.4471\sin(at)+1.2488\cos(at))e^{ct}\\
u_2^*(t)=& (40.3067\sin(at)-0.3685\cos(at))e^{-ct}
-(1.3063\sin(at)+0.6315\cos(at))e^{ct}\\
{u}^*_3(t)=&-3.315e^{-1.618t}-0.46938e^{1.618t}+1.1578e^{0.61805t}+1.6266e^{-0.61805t}\\
{u}^*_4(t)=&8.6789e^{-1.618t}-0.2901e^{1.618t}-0.4423e^{0.61805t}-2.6321e^{-0.61805t}\\
\end{aligned}
\end{equation*}
of problem \eqref{numerical-example}, where $a=0.34066$ and $ c=1.2712$.

It is easy to verify that Assumption 1.1,  Assumption 3.4 and Assumption 3.5 hold for the functions $g(x(T),\xi)=(x_1(T)+x_3(T), (\xi_1+1)(x_2(T)+x_4(T)))^\top$ and
$F(x(T),\xi)= (\xi_1^2 +\xi_2+1)\|x(T)\|^2, $
and random matrix $M(\xi)$ and vector $q(\xi)$.  Moreover, conditions of Theorem 2.1 hold, since ${\bf 0}\in {\cal V}$, $\mathbb{E}[M(\xi)]$ is a Z-matrix, and $K=\mathcal{B}(0,\sqrt{6})\subset \mathbb{R}^2.$

We apply the relaxation,  the SAA scheme and  the time-stepping method to problem \eqref{numerical-example}. We use Matlab built solver \emph{fmincon} to solve the discrete approximation problems of problem \eqref{numerical-example}. Setting  $\epsilon=0.00001$, for each pair $(\nu,h)$ with
$$\nu\in \{500, 1000, 2000, 3000, 4000\},  \quad h\in \{0.008, 0.005, 0.004, 0.002, 0.001 \},$$ we generate
i.i.d. samples $\Xi^{\nu,k}=\{\xi^k_1, \ldots, \xi^k_\nu\}, k=1,\ldots, 10000$.
We solve the discrete problem to find a solution
$(x^{\epsilon,\nu}_{h,k}, u^{\epsilon,\nu}_{h,k})$
 using each of the samples $\Xi^{\nu,k}$, $k=1,\ldots, 10000$.
Then we compute the optimal value of the discrete problem for each $k$
$$\Phi^{\nu,k}_h(x^{\epsilon,\nu}_{h,k},u^{\epsilon,\nu}_{h,k})=
\frac{1}{\nu}\sum^\nu_{i=1}F(x^{\epsilon,\nu}_{h,k}(T),\xi^k_i)-\frac{1}{2}(\|x^{\epsilon,\nu}_{h,k}\|^2_{L^2}
+\|u^{\epsilon,\nu}_{h,k}\|^2_{L^2}).$$
The errors between $\Phi(x^*,u^*)=25.17501124$ and the optimal value $\Phi^\nu_h(x^{\epsilon,\nu}_h,u^{\epsilon,\nu}_h)$ are estimated by
\[E_h^{\epsilon,\nu}=\frac{1}{10000}\sum^{10000}_{k=1}
(\Phi(x^*,u^*)-
\Phi^{\nu,k}_h(x^{\epsilon,\nu}_{h,k},u^{\epsilon,\nu}_{h,k}))^2.
\]

   The numerical results are shown in FIG. 1, which verify the convergence results in Sections 3-4.
\begin{figure}[htbp]
  \centerline{
\subfigure[]{\psfig{figure=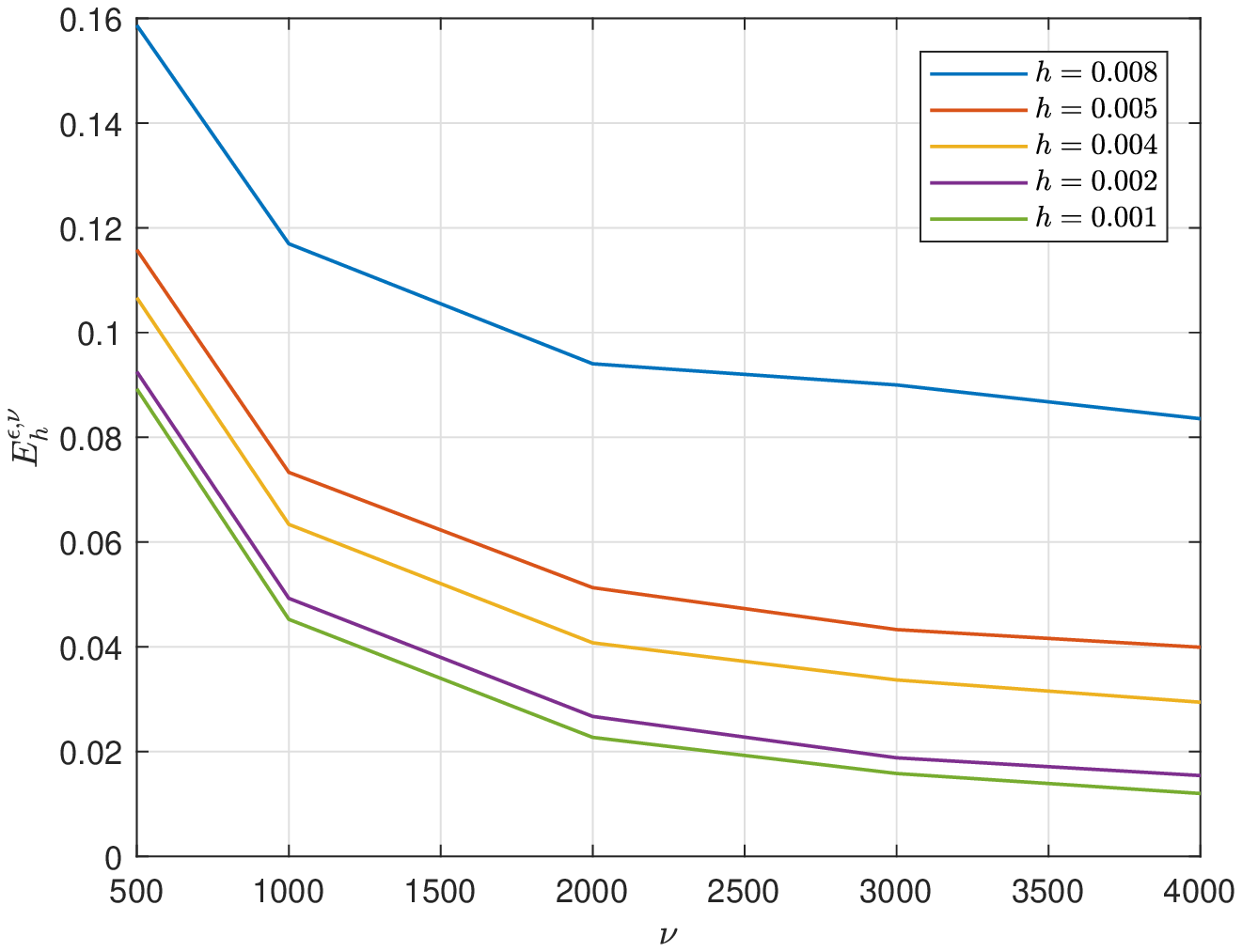,width=0.5\textwidth}}
\subfigure[]{\psfig{figure=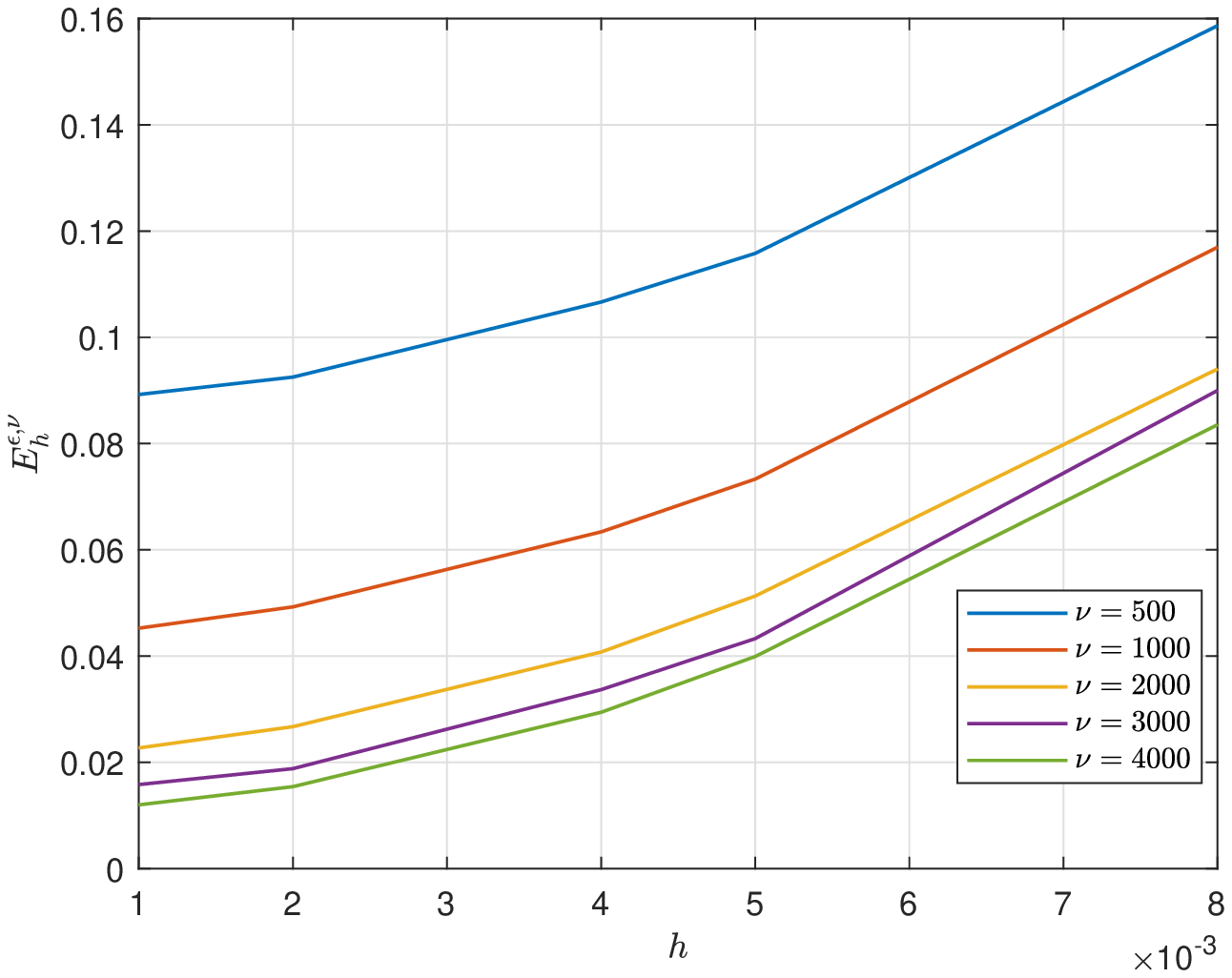,width=0.5\textwidth}}
}\caption{ Numerical errors between optimal values of (\ref{numerical-example}) and its discrete problems}
\end{figure}\label{figure}

\section{Conclusions}\label{se:conclusions}

In this paper, we study the optimal control problem with terminal stochastic linear complementarity constraints \eqref{OCTCCS}, and its relaxation-SAA problem \eqref{OCTCCS-sample} and the
relaxation-SAA-time stepping approximation problem \eqref{OCTCCS-sample-time}.
We prove the existence of feasible solutions and optimal solutions to problem \eqref{OCTCCS} in Theorem \ref{feasible-set} under the assumption
$E[M(\xi)]$ is a Z-matrix or an adequate matrix. Under the same assumptions of Theorem \ref{feasible-set}, we prove the existence of feasible solutions and optimal solutions to \eqref{OCTCCS-sample} and \eqref{OCTCCS-sample-time}.  We also show the convergent properties of these two discrete problems \eqref{OCTCCS-sample} and \eqref{OCTCCS-sample-time} as the relaxation parameter $\epsilon\downarrow0$, the sample size $\nu \rightarrow \infty$ and mesh size $h \downarrow0$.
Moreover, we provide asymptotics of the SAA optimal value and the error bound of the time-stepping method.
Problem \eqref{OCTCCS} extends optimal control problem with terminal deterministic linear complementarity constraints in \cite{FP2018}
to stochastic problems.  In \cite{FP2018}, Benita and Mehlita derived some stationary points and constraint qualifications under the assumption that the constrained LCP \eqref{LCP}-\eqref{LCPg} is solvable.  Theorem \ref{feasible-set} gives  sufficient conditions for the extension of solutions of \eqref{LCP}-\eqref{LCPg}.

\bibliographystyle{siamplain}
\bibliography{references}

\end{document}